\def\marker{\>\hbox{${\vcenter{\vbox{
    \hrule height 0.4pt\hbox{\vrule width 0.4pt height 5pt
    \kern6pt\vrule width 0.4pt}\hrule height 0.4pt}}}$}\>}

\documentclass[12pt]{article}
\usepackage{epsfig}
\usepackage{fullpage}
\usepackage{amsmath,amsfonts,amssymb,amsthm,graphics,amscd}
\usepackage{pgf,tikz}
\usepackage{graphicx}
\usepackage{subfigure}
\usepackage{amsmath,amsfonts,amssymb,amsthm,graphics,amscd,pst-plot,pstricks}
\usepackage{graphicx}

\newtheorem {Theorem}  {Theorem}[section]
\newtheorem {LEM}[Theorem]{Lemma}

\newtheorem {COR}[Theorem]{Corollary}

\theoremstyle{definition}

\newcommand{\D}{\Delta}

\newcommand{\phiv}{\varphi}

\newcommand{\CC}{\mathcal{C}}

\newcommand{\pbar}{\overline{\varphi}}

\DeclareMathOperator{\diam}{diam}
\usepackage{enumerate}
\usepackage[
pdfauthor={},
pdftitle={Precoloring extension of Vizing's Theorem for multigraphs},
pdfstartview=XYZ,
bookmarks=true,
colorlinks=true,
linkcolor=blue,
urlcolor=blue,
citecolor=blue,
bookmarks=false,
linktocpage=true,
hyperindex=true
]{hyperref}

\begin{document}
\title{\bf Precoloring extension of Vizing's Theorem for multigraphs}
\vspace{6mm}
\author{Yan Cao \thanks{yacao@mail.wvu.edu. Department of Mathematics, West Virginia University, Morgantown, WV 26506, USA.}\,\,, Guantao Chen \thanks{gchen@gsu.edu. Department of Mathematics and Statistics, Georgia State University, Atlanta, GA 30303, USA. Partially supported by NSF grant DMS-1855716.}\,\,, Guangming Jing \thanks{gjing@augusta.edu. Department of Mathematics, Augusta University, Augusta, GA 30912, USA. Partially supported by NSF grant DMS-2001130.}\,\,, Xuli Qi \thanks{ qixuli-1212@163.com. Department of Mathematics and Statistics, Georgia State University, Atlanta, GA 30303, USA. Partially supported by NSFC grants 11801135, 11871239 and 11771172.} \,\,, Songling Shan \thanks{sshan12@ilstu.edu. Department of Mathematics, Illinois State University, Normal, IL 61790, USA.}
}
\date{}
\maketitle

\begin{abstract}
Let $G$ be a graph with maximum degree $\Delta(G)$ and maximum multiplicity $\mu(G)$. Vizing and Gupta, independently, proved in the 1960s that the chromatic index of $G$ is at most $\Delta(G)+\mu(G)$.
The distance between two edges $e$ and $f$ in $G$ is the length of a shortest path connecting an endvertex of $e$ and an endvertex of $f$.  A {\em distance-$t$ matching} is a set of edges having pairwise distance at least $t$.
Edwards et al. proposed the following conjecture:  For any graph $G$, using the palette $\{1, \dots, \D(G)+\mu(G)\}$,
any precoloring on a distance-$2$ matching can be extended to a proper edge coloring of $G$.
Gir\~{a}o and Kang verified this conjecture for distance-$9$ matchings. In this paper, we improve the required distance from $9$ to $3$ for multigraphs $G$ with $\mu(G) \ge 2$. 	
	
\end{abstract}

\section{Introduction}
In this paper,  we  follow the book~\cite{SSTF} of
Stiebitz et al. for notation and terminologies. Graphs in this paper are finite, undirected, and without loops, but may have multiple edges.
Let $G=(V(G),E(G))$ be a graph, where $V(G)$ and $E(G)$ are respectively the vertex set and the edge set of  $G$.
Let $\D(G)$ and $\mu(G)$ be respectively the maximum degree and the maximum multiplicity of $G$.
Let $[k] :=\{1,  \dots, k\}$ be a palette of $k$ available colors.
A {\it $k$-edge-coloring} of $G$ is a map  that assigns to every edge of $G$ a color from the palette $[k]$  such that  no two adjacent edges receive the same color (the  edge coloring is also called {\it proper}).  Denote by $\CC^k(G)$ the set of all  $k$-edge-colorings of $G$. The {\it chromatic index\/} $\chi'(G)$ is  the least integer $k$ such that $\CC^k(G)\ne\emptyset$.
The distance between two edges $e$ and $f$ in $G$ is the length of a shortest path connecting  an endvertex of $e$ and an endvertex of $f$.
A {\it distance-$t$ matching} is a set of edges having pairwise distance at least $t$.  Following this definition, a matching is a distance-$1$ matching and an induced matching is a distance-$2$ matching. For a matching $M$, we use $V(M)$ to denote the set of vertices saturated by $M$.

In the 1960s, Vizing~\cite{Vizing-2-classes} and, independently,  Gupta~\cite{Gupta-67} proved that $\D(G)\le \chi'(G) \le \D(G)+\mu(G)$, which is commonly called Vizing's Theorem.
Using the palette $[\D(G)+\mu(G)]$, when can we extend a precoloring on a given edge set $F\subseteq E(G)$ to
a proper edge coloring of $G$?
Albertson and Moore~\cite{Albertson2001} conjectured that if $G$ is a simple graph,
using the palette $[\D(G)+1]$, any precoloring on a distance-$3$ matching can be extended to a proper edge coloring of $G$.  Edwards et al.~\cite{Edwards2018} proposed a stronger conjecture: {\it For any graph $G$, using the palette $[\D(G)+\mu(G)]$,
any precoloring on a distance-$2$ matching can be extended to a proper edge coloring of $G$.}
Gir\~{a}o and Kang~\cite{Kang2019} verified the conjecture of Edwards et al. for distance-$9$ matchings. In this paper, we improve the required distance from $9$ to $3$ for multigraphs with the maximum multiplicity at least $2$ as follows.

\begin{Theorem} \label{thm:main}
Let $G$ be a multigraph and $M$ be a distance-$3$ matching of $G$. If $\mu(G)\ge 2$ and $M$ is arbitrarily precolored from the palette $[\D(G) +\mu(G)]$, then there is a proper edge coloring of $G$ using colors from $[\D(G) +\mu(G)]$ that agrees with the precoloring on $M$.
\end{Theorem}

The {\em density} of a graph $G$, denoted $\Gamma(G)$,
is defined as $$\Gamma(G)=max\left\{\frac{2|E(H)|}{|V(H)|-1}:H\subseteq G,|V(H)|\ge3 \,\,  \text{and} \,\, |V(H)| \,\, \text{is odd} \right\} $$
if $|V(G)| \ge 3$ and $\Gamma(G) =0$ otherwise.
Note that for any $X\subseteq V(G)$ with odd $|X|\ge3$, we have $\chi'(G[X])\ge \frac{2|E(G[X])|}{|X|-1}$, where $G[X]$ is the subgraph of $G$ induced by $X$. Therefore, $\chi'(G) \ge \lceil \Gamma(G) \rceil$.
So, besides the maximum degree, the density provides another lower bound on the chromatic index of a graph.  In the 1970s, Goldberg~\cite{Goldberg1973} and Seymour~\cite{Seymour1979} independently conjectured that actually $\chi'(G) =\lceil \Gamma(G) \rceil$
provided  $\chi'(G) \ge \Delta(G)+2$.  The conjecture was commonly referred to as one of the most challenging problems in graph chromatic theory \cite{SSTF}. In joint work with  Zang, two authors of this paper, Chen and Jing gave a proof of the Goldberg-Seymour Conjecture recently \cite{CJZ19+}. We assume that the Goldberg-Seymour Conjecture is true in this paper.

We will prove Theorem~\ref{thm:main} in Section $4$. In Sections $2$ we introduce some new structural properties of dense subgraphs. In Section $3$  we define a general multi-fan and obtain some generalizations of Vizing's Theorem.

\section{Dense subgraphs}

Throughout the rest of this paper, we reserve the notation $\Delta$ and $\mu$
for  the maximum degree and the maximum multiplicity of the graph $G$, respectively. For $u\in V(G)$, let $d_G(u)$ denote the {\em degree} of  $u$ in $G$.
For a vertex set $N\subseteq V(G)$, let $G-N$ be the graph obtained from $G$ by deleting all the vertices in $N$ and edges incident with them. For an edge set $F\subseteq E(G)$, let
$G-F$ be the graph obtained from $G$ by deleting all the edges in $F$ but keeping their endvertices.
If $F=\{e\}$, we simply write $G-e$.
Similarly, we let $G+e$  be the graph obtained from $G$ by
adding the edge $e$ to $E(G)$.
For disjoint $X,Y\subseteq V(G)$,
$E_G(X,Y)$ is the set of edges of $G$ with one endvertex in $X$
and the other in $Y$.
If $X=\{x\}$, we simply write
$E_G(x,Y)$.
For  $X\subseteq V(G)$, the edge set $\partial_G(X) := E_G(X, V(G)\backslash X)$ is called the  {\it boundary} of $X$ in $G$.  For a subgraph $H$ of $G$, we simply write $\partial_G(H)$ for $\partial_G(V(H))$.

Let $G$ be a graph, $v\in V(G)$ and  $\varphi \in \CC^{k}(G)$ for some positive integer $k$. We define
$
\varphi(v)=\{\varphi(f):\text{$f \in E(G)$ and $f$ is incident with $v$}\} $ and $\pbar(v)=[k] \setminus\varphi(v).
$
We call $\varphi(v)$ the set of colors {\it present} at $v$ and $\pbar(v)$
the set of colors {\it missing} at $v$.
For a vertex set $X\subseteq V(G)$,  define  $\pbar(X)=\bigcup _{v\in X} \pbar(v)$.   A vertex set $X\subseteq V(G)$ is called
 {\it $\varphi$-elementary} if $\pbar(u)\cap \pbar(v)=\emptyset$
for every two distinct vertices $u,v\in X$.
The set $X$ is called  {\it $\varphi$-closed} if
each color on  edges from $\partial_G(X)$ is present at each vertex of $X$.
Moreover, the set $X$ is called  {\it strongly $\phiv$-closed}  if $X$ is $\phiv$-closed and colors on  edges from $\partial_G(X)$ are pairwise distinct.
For a subgraph $H$ of $G$, let $\phiv_H$ or $(\phiv)_H$ be the edge coloring of $G$ restricted on $H$.
We say a subgraph $H$ of $G$ is
$\phiv$-elementary, $\phiv$-closed and strongly $\phiv$-closed, if $V(H)$ is $\phiv$-elementary, $\phiv$-closed and strongly $\phiv$-closed, respectively.
Clearly, if $H$ is $\phiv_H$-elementary then  $H$ is $\phiv$-elementary, but the converse is not true as the edges in $\partial_G(H)$ are removed when we consider $\phiv_H$.

A subgraph $H$ of $G$ is  {\em $k$-dense }  if $|V(H)|$ is odd and $|E(H)|=(|V(H)|-1)k/2$.
Moreover, $H$ is a {\em maximal $k$-dense subgraph} if there does not exist a $k$-dense subgraph $H'$ containing $H$ as a proper subgraph.

\begin{LEM} {\em \cite{CCJ2021}} \label{MDS}
Given a graph $G$,  if $\chi'(G)=k \geq \Delta(G)+1$, then distinct maximal $k$-dense subgraphs of $G$ are pairwise vertex-disjoint.
\end{LEM}

\begin{LEM}\label{k-dense subgraph}
Let $G$ be a graph with  $\chi'(G) = k$ and $H$ be a $k$-dense subgraph of $G$. Then $H$ is an induced subgraph of $G$ with $\chi'(H)=\Gamma(H) = k$. Furthermore, for any coloring $\varphi\in \CC^{k}(G)$, $H$ is $\varphi_H$-elementary and strongly $\varphi$-closed.
\end{LEM}

\begin{proof}
Since $H$ is $k$-dense, by the definition,
$|E(H)|=\frac{|V(H)|-1}{2} k$. Thus $k\le \Gamma(H)\le\chi'(H)\le \chi'(G)=k$ implying $\chi'(H)=\Gamma(H)=k$. Thus $H$ is an induced subgraph of $G$, since otherwise there exists a subgraph $H'$ of $G$ with $V(H')=V(H)$ such that $\chi'(H')\ge\Gamma(H')>k$, a contradiction to $\chi'(H')\le \chi'(G)=k$.
Since $H$ has an odd order, the size of a maximum matching in $H$
has size at most $(|V(H)|-1)/2$. Therefore, under any $k$-edge-coloring $\varphi$ of $G$,
each color class in $H$ is a matching of size exactly $(|V(H)|-1)/2$.
Thus every color in $[k]$ is missing at exactly one vertex of $H$ or it appears exactly once in $\partial_G(H)$.
Consequently, $H$ is $\varphi_H$-elementary and strongly $\varphi$-closed.
\end{proof}

An edge $e$ of a graph $G$ is called a {\it k-critical edge}  if $k=\chi'(G-e) < \chi'(G)=k+1.$
A graph G is called {\it k-critical} if $\chi'(H)<\chi'(G)=k+1$ for each proper subgraph $H$ of $G$.
It is easy to see that a connected graph $G$ is $k$-critical if and only if every edge of $G$ is $k$-critical.
For $e\in E(G)$, let $V(e)$ denote the set of  the two endvertices of $e$. The following lemma is a consequent of the Goldberg-Seymour Conjecture.

\begin{LEM}\label{G-S-C}
	Let $G$ be a multigraph and $e\in E(G)$.  If $e$ is  a $k$-critical edge of $G$ and $k\ge \Delta(G) +1$, then
	$G-e$ has a $k$-dense subgraph $H$ containing $V(e)$ such that $e$ is also a $k$-critical edge of $H+e$.
\end{LEM}

\begin{proof}
Clearly, $\chi'(G) = k+1$ and $\chi'(G-e) =k$. By the assumption of the Goldberg-Seymour Conjecture,
$\chi'(G) =\lceil \Gamma(G) \rceil = k+1$. So, there exists a subgraph $H^*$ of odd order containing $e$  such that $|E(H^*)| > (|V(H^*)|-1)k/2$. On the other hand,
we have $\frac{2|E(H^*-e)|}{|V(H^*-e)|-1}\le\lceil \Gamma(H^*-e) \rceil \le\chi'(H^*-e) \le \chi'(G-e)=k$, which in turn gives $|E(H^*-e)| \le (|V(H^*)|-1)k/2$. Thus $|E(H^*-e)| =(|V(H^*)|-1)k/2$. Then $k\le \lceil \Gamma(H^*-e) \rceil\le\chi'(H^*-e)\le\chi'(G-e)=k$ and  $k+1\le\lceil \Gamma(H^*) \rceil \le\chi'(H^*)\le\chi'(G)=k+1$, which implies that $k=\chi'(H^*-e)<\chi'(H^*)=k+1$. Thus $H :=H^*-e$ is a $k$-dense subgraph containing $V(e)$, and $e$ is also a $k$-critical edge of $H+e$.
\end{proof}

The {\it diameter} of a graph $G$, denoted $\diam(G)$,  is the greatest distance between any pair of vertices  in $V(G)$.
\begin{LEM}\label{lem:elementarty-k-dense}
	Let $G$ be a multigraph with $\chi'(G)=k+1\geq \Delta(G)+2$ and $e$ be a $k$-critical edge of $G$.
	We have the following statements.
	
	$(a)$  $G-e$ has a unique maximal $k$-dense subgraph $H$ containing $V(e)$,  and $e$ is  also a $k$-critical edge of $H+e$.
	
	$(b)$ For any $\varphi\in \CC^{k}(G-e)$,  $H$ is $\varphi_H$-elementary and strongly $\varphi$-closed.
	
	$(c)$ If $\chi'(G)=\Delta(G)+\mu(G)$, then $\Delta(H+e)=\Delta(G)$, $\mu(H+e)=\mu(G)$ and   $\diam(H+e)\le \diam(H)\le 2$.
\end{LEM}

\begin{proof}  By Lemma \ref{G-S-C},
$G-e$ contains a $k$-dense subgraph $H$  containing $V(e)$ and $e$ is also a $k$-critical edge of $H+e$.
We may assume that $H$ is a maximal $k$-dense subgraph, and the uniqueness of $H$ is a direct consequence of Lemma \ref{MDS}. This proves $(a)$.
By applying Lemma \ref{k-dense subgraph} on $G-e$, we immediately have statement $(b)$.

For $(c)$, by $(a)$ and Vizing's Theorem, $\D(G)+\mu(G)=\chi'(G)=\chi'(H+e)\le \D(H+e)+\mu(H+e)\le\D(G)+\mu(G)$ implying that $\Delta(H+e)=\Delta(G)=\D$ and $\mu(H+e)=\mu(G)=\mu$. For any  $\varphi\in \CC^{k}(G-e)$,  $H$ is $\phiv_H$-elementary by $(b)$.
For any $x\in V(H)$, with respect to $\phiv_H$,
all the colors missing at other vertices of $H$ present at $x$. Note that $k=\D+\mu-1$. For each vertex $v\in V(H)$, we have that
$|\pbar_H(v)|=k-d_H(v) \ge k -\D = \mu -1$ if $v\notin V(e)$,  and $|\pbar_H(v)| = k -d_H(v)+1 \ge k -\D +1 \ge (\mu-1) +1$ if $v\in V(e)$. Denote $|V(H)|$ by $n$.
We then have
$d_H(x)\ge |\bigcup_{v\in V(H),v\ne x} \pbar_H(v)| \ge (k-\Delta)(n-1)+1=(\mu-1)(n-1)+1$.

Since $\mu(H) \le \mu(G) = \mu$,
we get $|N_H(x)|\ge \frac{d_H(x)}{\mu} \ge \frac{(\mu-1)(n-1)+1}{\mu}$, where $N_H(x)$ is the neighbor set of $x$ in $H$.
Since $\mu \ge 2$, we have  $\frac{(\mu-1)(n-1)+1}{\mu}
\ge \frac{n}{2}$.  Hence, every vertex in $H$ is adjacent to at least half vertices in $H$. Consequently, every two vertices of $H$ share a common neighbor, which in turn gives  $\diam(H) \le 2$. This proves $(c)$.
\end{proof}

An {\em $i$-edge} is an edge colored with the color $i$. The following technical lemma will be used several times  in our proof.

\begin{LEM}\label{lem-consisting} Let $G$ be a graph with $\chi'(G)=k$ and $H$ be a $k$-dense subgraph of $G$.
Let $\psi$ and $\varphi$ respectively be
$k$-edge-colorings of $H$ and $G-E(H)$ such that colors on edges in $\partial_G(H)$ are pairwise distinct under $\varphi$. The following two statements hold.

$(a)$ If $k \ge \D(G)$, then by renaming color classes of $\psi$ on $E(H)$,  we can obtain a (proper) $k$-edge-coloring  of $G$ by combining $\phiv$ and the modified coloring based on $\psi$.

$(b)$ For any fixed color $i\in[k]$, if $k \ge \D(G) +1$, then by renaming other color classes of $\psi$ on $E(H)$  we can obtain a coloring of $G$  such that all color classes are matchings except the $i$-edges. The only exception is as follows: exactly one $i$-edge from $E(H)$ and exactly one $i$-edge from $\partial_G(H)$ share an endvertex.
\end{LEM}

\begin{proof}
Since   $\chi'(G) = k$ and $H$ is $k$-dense,  $\chi'(H) = k$ and $H$ is $\psi$-elementary by Lemma \ref{k-dense subgraph}.  This following fact will be used
to combine an edge coloring of $H$ and an edge coloring of $G-E(H)$
into an edge coloring of $G$:  for any distinct $u,v\in V(H)$, $\overline{\psi}(u)\cap \overline{\psi}(v)=\emptyset$, and  no two colors on edges in $\partial_G(H)$ under $\varphi$ are the same.

For $(a)$, we have
$|\overline{\psi}(v)| = k -d_H(v) \ge \D(G) - d_H(v) \ge d_{G-E(H)}(v)=|\phiv(v)|$ for each $v\in V(H)$.
So, by renaming color classes of $\psi$ on $E(H)$, we may assume that $\phiv(v)\subseteq\overline{\psi}(v)$ for each $v\in V(H)$. The combination of $\phiv$  and the modified coloring based on $\psi$  gives a desired proper edge coloring of $G$.

For $(b)$, under the condition $k \ge \D(G) +1$, we have  $|\overline{\psi}(v)| = k -d_H(v) \ge \D(G) +1 - d_H(v) \ge d_{G-E(H)}(v) +1=|\phiv(v)|+1$ for each $v\in V(H)$. So
$|\overline{\psi}(v)\backslash \{i\}| \ge |\phiv(v)\backslash \{i\}|$.  Notice that when $i\in \overline{\psi}(v) \cap \pbar(v)$,
we need $|\overline{\psi}(v)| -1 \ge |\phiv(v)|$ to ensure the inequality above, where the condition  $k \ge \D(G) +1$ is applied. By renaming color classes of $\psi$ on $E(H)$ except the $i$-edges (keeping all $i$-edges unchanged and other color classes not renamed by $i$), we may assume that $\phiv(v) \backslash \{i\} \subseteq \overline{\psi}(v)$ for each $v\in V(H)$. Again, the combination of  $\phiv$ and the modified coloring based on $\psi$  gives a desired  coloring of $G$. The only case that the set of $i$-edges
is not a matching is when exactly one $i$-edge from $E(H)$ and exactly one $i$-edge from $\partial_G(H)$ share an endvertex, since colors on edges in $\partial_G(H)$ are pairwise distinct under $\varphi$.
\end{proof}

\section{Refinements of multi-fans and some consequences}

We first recall Kempe-chains and related terminologies.
Let $\phiv$ be a $k$-edge-coloring of $G$ using the palette
$[k]$.
Given two distinct colors $\alpha,\beta$, an {\it $(\alpha, \beta)$-chain} is
a component of  the  subgraph induced by
edges assigned color $\alpha$ or $\beta$ in $G$, which is either an even cycle or a path.
We call the operation that swaps the colors $\alpha$ and $\beta$
on an $(\alpha,\beta)$-chain  the {\it Kempe change}.
Clearly, the resulting coloring after a Kempe change is
still a (proper)  $k$-edge-coloring. Furthermore, we say that a chain has  {\em endvertices} $u$ and $v$ if the chain is a path connecting vertices $u$ and $v$. For a vertex $v\in V(G)$, we denote by $P_v(\alpha,\beta)$ the unique $(\alpha,\beta)$-chain containing the vertex $v$. For two vertices $u$, $v\in V(G)$, the two chains $P_u(\alpha,\beta)$ and $P_v(\alpha,\beta)$ are either identical or  disjoint.
More generally, for an $(\alpha,\beta)$-chain, if it is a path and it contains two vertices $a$ and $b$, we let $P_{[a,b]}(\alpha, \beta)$ be its subchain  with endvertices $a$ and $b$.
The operation of swapping colors $\alpha$ and $\beta$ on the subchain $P_{[a,b]}(\alpha, \beta)$ is still called a  Kempe change,
but the resulting coloring may no longer be a proper edge coloring.

Let $G$ be a graph with an edge $e\in E_G(x,y)$, and $\varphi$ be a proper edge coloring of $G$ or $G-e$. A sequence
$F=(x,e_0,y_0,e_1, y_1, \ldots, e_p, y_p)$ with integer $p\ge 0$ consisting of vertices
and distinct edges is called a (general)
{\em  multi-fan} at $x$ with respect to $e$ and $\varphi$ if $e_0=e$,  $y_0=y$, for each $i\in [p]$, $e_i\in E_G(x,y_i)$  and
there is a vertex $y_j$ with $0\leq j\leq i-1$ such that $\phiv(e_i) \in \overline{\varphi}(y_j)$.
Notice that the  definition of a multi-fan in this paper is  slightly general than the one in \cite{SSTF} since the edge $e$ may be colored in $G$.
We say a multi-fan $F$ is {\em maximal} if  there is no  multi-fan containing $F$ as a proper subsequence. Similarly, we say a multi-fan $F$ is {\em maximal without any $i$-edge} if  $F$ does not contain any $i$-edge and there is no multi-fan without any $i$-edge containing $F$ as a proper subsequence. The set of vertices and edges contained in $F$ are denoted by $V(F)$ and $E(F)$, respectively. Let $e_G(x,y)=|E_G(x,y)|$ for $x,y\in V(G)$. Note that a multi-fan may have repeated vertices.  By $e_F(x,y_i)$ for some $y_i\in V(F)$ we mean the number of edges joining $x$ and $y_i$ in $F$.

Let $s\ge 0$ be an integer.  A {\em linear sequence} $S=(y_0, e_1,y_1,\dots,e_s,y_s)$ at $x$ from $y_0$ to $y_s$ in $G$ is a sequence consisting of distinct vertices and distinct edges such that $e_i\in E_G(x,y_i)$ for $i\in[s]$ and $\varphi(e_i)\in\overline\varphi(y_{i-1})$ for $i\in [s]$. Clearly for any $y_j\in V(F)$, the multi-fan $F$ contains a linear sequence at $x$ from $y_0$ to $y_j$ (take a shortest sequence $(y_0, e_1,y_1,\dots,e_j,y_j)$ of vertices and edges with the property that $e_i\in E_G(x,y_i)\cap E(F)$ for $i\in[j]$ and $\varphi(e_i)\in\overline\varphi(y_{i-1})$ for $i\in [j]$).  
The following local edge recoloring operation will be used in our proof.
A {\em shifting} from $y_i$ to $y_j$ in the linear sequence $S$ is an operation that replaces the current color of $e_t$ by the color of $e_{t+1}$ for each $i\le t\leq j-1$ with $1\leq i< j\leq s$. Note that the shifting does not change the color of $e_j$, where $e_j$ joins $x$ and $y_j$, so the resulting coloring after a shifting is   not  a proper coloring. In our proof we will uncolor or recolor the edge $e_j$ to make the resulting coloring proper. We also denote by
$V(S)$ and $E(S)$ the set of vertices and the set of edges contained in the linear sequence $S$, respectively.

\begin{LEM} {\em \cite{Goldberg1974,SSTF}} \label{MF}
	Let $G$ be a graph, $e\in E_G(x,y)$ be a $k$-critical edge and $\phiv\in \CC^k(G-e)$  with $k \ge \D(G)$. Let $F=(x,e,y_0,e_1, y_1, \ldots, e_p, y_p)$ be a multi-fan at $x$ with respect to $e$ and $\varphi$, where $y_0=y$. Then the following statements hold.
	
	$(a)$ $V(F)$ is $\varphi$-elementary, and each edge in $E(F)$ is a $k$-critical edge of $G$.
	
	$(b)$ If $\alpha \in \overline{\phiv}(x)$ and $\beta \in \overline{\phiv}(y_i)$ for $0\le i\le p$, then $P_x(\alpha,\beta)=P_{y_i}(\alpha,\beta)$.
	
	$(c)$ If $F$ is a maximal multi-fan at $x$ with respect to $e$ and $\phiv$,
 then $x$ is adjacent in $G$ to at least $\chi'(G)-d_G(y)-e_G(x,y)+1$ vertices $z$ in $V(F)\backslash \{x,y\}$ such that $d_G(z)+e_G(x,z)=\chi'(G)$.
\end{LEM}

A {\em $\Delta$-vertex} in $G$ is a vertex with degree exactly $\Delta$ in $G$. A {\em $\Delta$-neighbor} of a vertex $v$  in $G$ is a neighbor of $v$ that is a $\Delta$-vertex in $G$.
\begin{LEM}  \label{MD}
Let $G$ be a  multigraph with maximum degree
$\Delta$ and maximum multiplicity $\mu\ge2$.  Let $e\in E_G(x,y)$  and $k=\Delta+\mu-1$.

Assume that   $\chi'(G)=k+1$,  $e$ is $k$-critical and $\phiv\in \CC^{k}(G-e)$. Let $F=(x,e,y_0,e_1, y_1, \ldots,\\ e_p, y_p)$ be a multi-fan at $x$ with respect to $e$ and $\varphi$, where $y_0=y$.
Then the following statements hold.
	
	$(a)$  If $F$ is maximal,  then $x$ is adjacent in $G$ to at least $\Delta+\mu-d_G(y)-e_G(x,y)+1$ vertices $z$ in $V(F)\backslash \{x,y\}$ such that $d_G(z)=\Delta$ and $e_G(x,z)=\mu$.
	
	$(b)$  If $F$ is maximal, $d_G(y)=\Delta$ and  $x$ has only one $\Delta$-neighbor $z'$ in $G$ from $V(F)\backslash \{x,y\}$, then $e_F(x,z)=e_G(x,z)=\mu$ for all $z\in V(F)\backslash \{x\}$ and $d_G(z)=\Delta-1$ for all $z\in V(F)\backslash \{x,y,z'\}$.
	
	$(c)$  For $i\in [k]$ and $i\notin \overline{\phiv}(y)$, if $F$ is maximal without any $i$-edge, then $F$ not containing any $\Delta$-vertex of $G$ from $V(F)\backslash \{x,y\}$ implies that $d_G(y)=\D$, and there exists a vertex $z^*\in V(F)\backslash \{x,y\}$ with $i\in \overline{\phiv}(z^*)$ such that $d_G(z^*)=\Delta-1$.
	
Assume that $\chi'(G)=k$, $\phiv\in \CC^{k}(G)$ and $V(G)$ is  $\phiv$-elementary. Then the following statement holds.
	
	$(d)$ If a multi-fan  $F'$ is maximal at $x$ with respect to $e$ and $\varphi$  in $G$, then $x$ having no $\Delta$-neighbor in $G$ from $V(F')$ implies that $d_G(z)=\Delta-1$ for all $z\in V(F')\backslash \{x\}$ and every edge in $F'$  is colored by a missing color at some vertex in $V(F')$.
Furthermore, for $i\in [k]$ and $\varphi(e)\notin \overline{\phiv}(V(F'))$, if $F'$ is maximal without any $i$-edge, then $F'$ not containing any $\Delta$-vertex in $G$ from $V(F')\backslash \{x\}$ implies that there exists a vertex $z^*\in V(F')\backslash \{x\}$ with $i\in \overline{\phiv}(z^*)$ such that $d_G(z^*)=\Delta-1$.
\end{LEM}

\begin{proof}
	For statements $(a)$, $(b)$  and $(c)$,
	$V(F)$ is $\phiv$-elementary by Lemma \ref{MF}$(a)$. Statement  $(a)$ holds easily by Lemma \ref{MF}$(c)$.
	Assume that there are $q$ distinct vertices in $V(F)\backslash \{x\}$.

	For $(b)$, we have
	\begin{eqnarray*}
		q\mu &\ge&\sum_{z\in V(F)\backslash \{x\}}e_G(x,z)\ge\sum_{z\in V(F)\backslash \{x\}}e_F(x,z)=1+\sum_{z\in V(F)\backslash \{x\}}|\overline{\phiv}(z)|\\
		&\ge&1+(k-\Delta+1)+(k-\Delta)+(q-2)(k-\Delta+1)=q(k-\Delta+1)= q\mu,
	\end{eqnarray*}
as $|\overline{\phiv}(y)|=k-\Delta+1$, $|\overline{\phiv}(z')|=k-\Delta$ and $|\overline{\phiv}(z)|\ge k-\Delta+1$ for $z\in V(F)\backslash \{x,y,z'\}$.
Therefore, $e_F(x,z)=e_G(x,z)=\mu$ for each $z\in V(F)\backslash \{x\}$ and $d_G(z)=\Delta-1$ for each $z\in V(F)\backslash \{x,y,z'\}$. This proves $(b)$.
	
	Now for $(c)$, we must have that there exists a vertex $z^*\in V(F)\backslash \{x,y\}$ with $i\in \overline{\phiv}(z^*)$, since otherwise by (a), $x$ has at least one $\Delta$-neighbor in $G$ from $V(F)\backslash \{x,y\}$, a contradiction. Since  $V(F)$ is $\phiv$-elementary, $x$ must be incident with an $i$-edge. Since now there is no $i$-edge in $F$ and $i\in\overline{\phiv}(z^*)$, we have
	\begin{eqnarray*}
		q\mu &\ge&\sum_{z\in V(F)\backslash \{x\}}e_G(x,z)\ge \sum_{z\in V(F)\backslash \{x\}}e_F(x,z)=1+(|\overline{\phiv}(z^*)|-1)+\sum_{z\in V(F)\backslash \{x,z^*\}}|\overline{\phiv}(z)|\\
		&\ge&1+k-\Delta+(q-1)(k-\Delta+1)=q(k-\Delta+1)= q\mu.
	\end{eqnarray*}
	Therefore, $d_G(y)=\Delta$ and $d_G(z)=\Delta-1$ for each $z\in V(F)\backslash \{x,y\}$. This proves $(c)$.
	
	Statement $(d)$ follows from similar calculations as in the proof of $(b)$ and $(c)$.
\end{proof}

Let $G$ be a graph with maximum degree $\D$ and maximum multiplicity $\mu$. Berge and Fournier~\cite{Berge1991A} strengthened
the classical Vizing's Theorem by showing that
if $M^*$ is a maximal matching of $G$, then $\chi'(G- M^*) \le \Delta+\mu -1$. An edge $e\in E_G(x,y)$ is
{\it fully $G$-saturated} if $d_G(x) = d_G(y) = \D$ and $e_G(x,y) = \mu$. For every graph $G$ with $\chi'(G) = \D + \mu$, observe that $G$ contains a  $(\Delta+\mu-1)$-critical subgraph $H$ with $\chi'(H)=\D + \mu$ and $\D(H)=\D$ by Lemma \ref{lem:elementarty-k-dense}$(c)$, and $G$ contains at least two fully $G$-saturated edges  by Lemma \ref{MD}$(a)$.

Stiebitz et al.[Page 41 Statement (a), \cite{SSTF}] obtained the following generalization of Vizing's Theorem with an elegant short proof:
{\em Let $G$ be a graph and let $k \ge \D + \mu$ be an integer. Then there is a
$k$-edge-coloring $\phiv$ of $G$  such that every edge $e$  with $\phiv(e) =k$ is fully G-saturated}.
We observe that their proof actually gives a slightly stronger result which also
generalizes the Berge-Fournier theorem as follows.

\begin{LEM}\label{lem:BF-gen}
Let $G$ be a graph, and $M$ and $M'$ be two vertex-disjoint matchings of $G$. If every edge of $M'$ is fully $G$-saturated  and $M'$ is maximal subject to this property, then $\chi'(G- (M\cup M')) \le\D(G) + \mu(G) -1$.
\end{LEM}

\begin{proof} Let $G' = G-(M\cup M')$. Note that every vertex $v\in V(M\cup M')$ has $d_{G'}(v) \le \D-1$.  By the maximality of $M'$, $G-V(M\cup M')$ contains
no fully $G$-saturated edges. So, $G'$ does not have a fully $G$-saturated edge. By the observation of graphs with chromatic index $\D + \mu$ and Lemma \ref{MD}$(a)$,  $\chi'(G') \le \D + \mu -1$, since otherwise $\D(G')=\D$ and there exist at least two fully $G$-saturated edges in one multi-fan centered at a $\D$-vertex, a contradiction.
\end{proof}

Lemma \ref{lem:BF-gen} has the following consequence.
\begin{COR}\label{cor:BF-gen}
Let $G$ be a graph.  If $M$ is a matching such that every edge in $M$ is fully $G$-saturated and $M$ is maximal subject to this property,  then $\chi'(G-M) \le\D(G) + \mu(G) -1$.
\end{COR}

We strengthen Lemma \ref{lem:BF-gen} for multigraphs $G$ with $\mu(G) \ge 2$  as follows.
\begin{LEM}\label{lem:BF-gen+}
For a fixed matching  $M$ of a graph $G$, if $\mu(G) \ge 2$ and $\chi'(G-M)=\D(G)+\mu(G)$, then there exists a matching $M^*$ of $G-V(M)$ such that $\chi'(G-(M\cup M^*)) = \D(G)+\mu(G)-1=:k$ and every edge $e\in M^*$ is $k$-critical and fully $G$-saturated in the graph $H_e+e$, where $H_e$ is the unique maximal $k$-dense subgraph of $G-(M\cup M^*)$ containing $V(e)$.
\end{LEM}

\begin{proof}
Let $M^*$ be a matching of $G-V(M)$ consisting of fully $G$-saturated edges. We further choose $M^*$ such that $M^*$ is maximal. By Lemma \ref{lem:BF-gen}, $\chi'(G-(M\cup M^*))= k$. If there exists $e\in M^*$ such that $\chi'(G-(M\cup M^*\backslash \{e\}))=k$, we remove $e$ out of $M^*$. Thus we may assume that for each $e\in M^*$, $\chi'(G-(M\cup M^*\backslash \{e\}))=k+1$, i.e., each $e$ is a $k$-critical edge of  $G-(M\cup M^*\backslash \{e\})$. By Lemma \ref{lem:elementarty-k-dense}$(a)$, there exists a unique maximal $k$-dense subgraph $H_e$ of $G-(M\cup M^*)$ such that $V(e)\subseteq V(H_e)$ and $e$ is  also a $k$-critical edge of $H_e+e$. Notice that $\Delta(H_e+e)=\Delta$ and $\mu(H_e+e)=\mu$ by Lemma \ref{lem:elementarty-k-dense}$(c)$. It is now only  left to show that each $e\in M^*$ is full $G$-saturated in the graph $H_e+e$.
Suppose on the contrary that there exists $e\in M^*$ such that $e$ is not fully $G$-saturated in $H_e+e$.

Since $e$ is a $k$-critical edge of $G-(M\cup M^*\backslash \{e\})$, we let $\varphi \in \CC^k(G-(M\cup M^*))$.  By Lemma~\ref{k-dense subgraph}, $H_e$ is $\phiv_{H_e}$-elementary and
strongly  $\varphi$-closed.
Let $V(e) = \{x, y\}$ and $F_x$ be a maximum multi-fan at $x$ with respect to $e$ and $\phiv_{H_e}$.
By Lemma \ref{MD}$(a)$, $x$ has a $\Delta$-neighbor, say $x_1$, in $H_e$ from $V(F_x)\backslash\{x,y\}$. By Lemma \ref{MF}$(a)$, the edge $e_{xx_1}\in E_G(x,x_1)$ in $F_x$ is also a $k$-critical edge of  $H_e +e$. By Lemma \ref{MD}$(a)$ again, in a maximum multi-fan at $x_1$ there exists  a fully $G$-saturated edge $e'$.
Let $M' = (M^*\backslash \{e\})\cup \{e'\}$.  Since every vertex of $V(M\cup M^*)$ has degree less than $\D$ in $G-(M\cup M^*)$,  it follows that $M\cup M'$ is a matching of $G$.
Let $H_{e'} =H_e+e-e'$.  Clearly, $H_{e'}$ is also $k$-dense.  Applying Lemma \ref{MF}$(a)$, we see that  $e'$ is also a $k$-critical edge of $H_e +e$. Thus
$\chi'(H_{e'})= k$ and $H_{e'}$ is also an induced subgraph of $G- (M\cup M')$ by Lemma \ref{k-dense subgraph}. Moreover,  $H_{e'}$ is  a  maximal $k$-dense subgraph of $G- (M\cup M')$, since otherwise there exists a $k$-dense subgraph $H'$ containing $H_{e'}$ as a proper subgraph which implies that the  $k$-dense subgraph $H'+e'-e$ is also a $k$-dense subgraph containing $H_e$ as a proper subgraph in $G-(M\cup M^*)$, a contradiction to the maximality of $H_e$. As $H_e$ is strongly $\varphi$-closed,  colors on edges of $\partial_{G-(M\cup M')}(H_{e'})=\partial_{G-(M\cup M^*)}(H_e)$ are pairwise distinct. Applying Lemma \ref{lem-consisting}$(a)$ on any $k$-edge-coloring of $H_{e'}$ and the $k$-edge-coloring of $G-(M\cup M'\cup E(H_{e'}))$,
we have $\chi'(G- (M\cup M')) = k$.
In order to claim that we can replace $e$ by $e'$
in $M^*$, and so repeat the same process for every edge  $f$ of $M^*$ that is not fully $G$-saturated in $H_f+f$ ($H_f$ is the maximal $k$-dense subgraph  of $G-(M\cup M^*)$ with $V(f)\subseteq V(H_f)$,  we discuss that this replacement will not affect the properties of other edges
in $M^*$ as follows.

By Lemmas~\ref{MDS} and \ref{k-dense subgraph},
maximal $k$-dense subgraphs
of $G-(M\cup M^*)$ are induced  and vertex-disjoint.
Thus for any $f\in M^*\backslash \{e\}$, either $V(H_f)\cap V(H_{e})=\emptyset$ or $H_f = H_{e}$.  If $V(H_f)\cap V(H_{e})=\emptyset$,  then $H_f$  is still  the induced maximal $k$-dense subgraph of  $G-(M\cup M')$ containing $V(f)$ and $f$ is $k$-critical in $H_f+f$. If $H_f=H_e$,
then as $H_{e'}$ is an induced maximal $k$-dense subgraph of  $G-(M\cup M')$ with $V(H_e)=V(H_{e'})$, it follows that $H_f+e-e'=H_{e'}$ is
the maximal $k$-dense subgraph of $G-(M\cup M')$   containing $V(f)$ and $f$ is $k$-critical in  $H_f+e-e'+f$  by Lemma~\ref{lem:elementarty-k-dense}(a).
As $V(f)\cap V(e)=\emptyset$ and $V(f)\cap V(e')=\emptyset$,
the property that whether or not $f$ is   fully $G$-saturated in  $H_f+f$ is not changed after replacing $e$ by $e'$ in $M^*$.
Therefore, by repeating the replacement process as for the edge $e$  above for every edge  $f$ of $M^*$ that is not fully $G$-saturated in $H_f+f$, we may assume that each edge $e\in M^*$  is  fully $G$-saturated in $H_e+e$. The proof is completed.
\end{proof}

\section{Proof of Theorem~\ref{thm:main}}

We rewrite Theorem \ref{thm:main} as follows.
\setcounter{section}{1}
\begin{Theorem}
Let $G$ be a multigraph with  $\mu(G) \ge 2$. Using palette $[\D(G) +\mu(G)]$, any precoloring on a distance-$3$ matching $M$ in $G$ can be extended to a proper edge coloring of  $G$.
\end{Theorem}

\setcounter{section}{4}

\begin{proof}
Let $k=\Delta+\mu-1$
 and  $\Phi: M\rightarrow [\Delta+\mu]$ be a given precoloring on $M$. Note that $\chi'(G-M)\le k+1$ by Vizing's Theorem.
The conclusion of Theorem~\ref{thm:main} holds easily if $\chi'(G-M) \le k$ with the reason as follows.
For any $k$-edge-coloring  $\psi$ of $G-M$,
if there exists $e\in E(G-M)$ such that $e$  is adjacent in $G$ to an edge $f\in M$
and $\psi(e)=\Phi(f)$, we recolor each such $e$ with the color $\Delta+\mu$ and  get a new coloring $\psi'$ of $G-M$.
Under  $\psi'$,
the edges colored by $\Delta+\mu$ form a matching in $G$ since  $M$ is a distance-$3$ matching. Thus the combination of $\Phi$ and $\psi'$ is a
$(k+1)$-edge-coloring of $G$. Therefore, in the remainder of the proof, we assume $\chi'(G-M)=k+1$.

Let $M_{\Delta+\mu}$ be the set of edges precolored with $\Delta+\mu$ in $M$ under $\Phi$. For  any
matching $M^*\subseteq G-V(M)$   and  any  $(k+1)$-edge-coloring or $k$-edge-coloring   $\varphi$ of $G-(M\cup M^*)$,
denote the $\Delta+\mu$ color class of   $\varphi$ by $E^\varphi_{M^*}$. In particular, $E^\varphi_{M^*}=\emptyset$ if $\varphi$ is a $k$-edge-coloring.
We introduce the following notation.
For  $f\in E_G(u,v) \cap M$, if there exists $f_1\in E(G-(M\cup M^*))$ such that
$V(f_1)\cap V(f)=\{u\}$ and
$\varphi(f_1)=\Phi(f)$,  we call $f$ {\bf T1-improper }(Type 1 improper) at $u$
if $V(f_1)\cap V(M^*)=\emptyset$, and {\bf T2-improper }(Type 2 improper) at $u$
if $V(f_1)\cap V(M^*) \ne \emptyset$.  If $f$ is T1-improper or T2-improper at $u$,
we say that $f$ is {\bf improper} at $u$.
Define
\begin{eqnarray*}
	E_1(M^*,\varphi)&=&\{ f_1\in E(G-(M\cup M^*)):  \text{$f_1$ is adjacent in $G$ to a T1-improper edge}\}, \\
	E_2(M^*,\varphi)&=&\{ f_1\in E(G-(M\cup M^*)):  \text{$f_1$ is adjacent in $G$ to a T2-improper edge}\}.
\end{eqnarray*}
Observe that $E_1(M^*,\varphi) \cup E_2(M^*,\varphi)$ is a matching since $M$ is a distance-$3$ matching in $G$.
We call the  triple $(M^*,E^\varphi_{M^*},\varphi)$  {\bf prefeasible}  if the following conditions are satisfied:
\begin{enumerate}[(a)]
	\item $M_{\Delta+\mu}\cup M^*\cup E^\varphi_{M^*}$ is a matching;
	\item for each  $e\in M^*$ such that $e$ is adjacent in $G$ to an edge of $E_2(M^*,\varphi)$, $e$ is $k$-critical and fully $G$-saturated in the graph $H_e+e$, where $H_e$ is the unique maximal $k$-dense subgraph of $G-(M\cup M^*)$ containing $V(e)$;
	\item  the colors on edges of $\partial_{G-(M\cup M^*)}(H_e)$ are all  distinct under $\varphi$.
\end{enumerate}

Let $(M^*,E^\varphi_{M^*},\varphi)$  be a prefeasible triple.
Since $M\cup M^*$ is a matching in $G$,
if  $(M^*,E^\varphi_{M^*},\varphi)$ also satisfies  {\it Condition $(d)$}: $|E_1(M^*,\varphi)|=|E_2(M^*,\varphi)|=0$,
then by assigning  the color $\Delta+\mu$ to all  edges of $M^*$, we obtain a (proper) $(k+1)$-edge-coloring  of $G$,   where the $(k+1)$-edge-coloring is  the combination of the  precoloring $\Phi$ on $M$,   the coloring  using the color $\Delta+\mu$ on $M^*$,  and the coloring $\varphi$ of $G-(M\cup M^*)$. Thus we define a  {\bf feasible} triple $(M^*,E^\varphi_{M^*},\varphi)$  as one that satisfies   Conditions $(a)$-$(d)$.

The rest of the proof is devoted to showing the existence of a feasible triple $(M^*,E^\varphi_{M^*},\varphi)$ of $G$.
Our main strategy is to first fix a particular prefeasible triple $(M^*_0,E^{\varphi_0}_{M^*_0},\varphi_0)$, then  modify it step by step into a feasible triple $(M^*,E^\varphi_{M^*},\varphi)$. In particular, we will choose $M^*_0$ and $\varphi_0$
such that $E_{M^*_0}^{\varphi_0}=\emptyset$. At the end,
when we modify $\varphi_0$ into $\varphi$, we will ensure that the $\Delta+\mu$ color class of $G$ is $M_{\Delta+\mu} \cup M^*\cup E_1(M_0^*,\varphi_0) \cup E_2(M_0^*,\varphi_0)$. The process is first to modify $M_0^*$ and $\varphi_0$ at the same time to deduce the number of T2-improper edges.

By Lemma \ref{lem:BF-gen+},  there exists a matching $M^*_0$ of $G-V(M)$ such that $\chi'(G-(M\cup M^*_0)) = k$ and each edge $e\in M^*_0$ is $k$-critical and fully $G$-saturated  in $H_e+e$, where $H_e$ is the unique maximal $k$-dense subgraph of $G-(M\cup M^*_0)$ containing $V(e)$. By Lemmas~\ref{MDS} and \ref{k-dense subgraph}, $H_e$ is induced in  $G-(M\cup M^*_0)$ with $\chi'(H_e)=k$, and $H_e$ and $H_{e'}$ are either identical or vertex-disjoint  for any $e'\in M_0^*\setminus \{e\}$.  Moreover, by Lemma \ref{lem:elementarty-k-dense},  $\diam (H_e+e)\le\diam (H_e)\le2$, and  $H_e$ is $(\varphi_0)_{H_e}$-elementary and strongly $\varphi_0$-closed in  $G-(M\cup M^*_0)$. As $\chi'(G-M)=k+1$, we  have   $|M^*_0|\ge 1$. Let $\varphi_0$ be a $k$-edge-coloring of  $G-(M\cup M^*_0)$.  Thus $E^{\varphi_0}_{M_0^*}=\emptyset$. Obviously,  the triple  $(M^*_0,\emptyset,\varphi_0)$ is prefeasible,  which we take  as our initial triple.

For $(M^*_0,\emptyset,\varphi_0)$,  if $|E_1(M_0^*,\varphi_0)|=|E_2(M_0^*,\varphi_0)|=0$,  then we are done. If $|E_1(M_0^*,\varphi_0)|\ge 1$ and $|E_2(M_0^*,\varphi_0)|=0$,  then we recolor each edge in $E_1(M_0^*,\varphi_0)$  with the color $\Delta+\mu$ to produce a $(k+1)$-edge-coloring $\varphi_1$ of $G-(M\cup M^*_0)$, since $E_1(M_0^*,\varphi_0)$ is a matching. Then  as $|E_1(M_0^*,\varphi_1)|=|E_2(M_0^*,\varphi_1)|=0$ and
$M_{\Delta+\mu}\cup M_0^*\cup E^{\varphi_1}_{M_0^*}$  is a matching, it follows that   the new  triple  $(M^*_0,E_1(M_0^*,\varphi_0),\varphi_1)$ is feasible. Then we are also done.
	
Therefore, we assume that $|E_1(M_0^*,\varphi_0)|\ge 0$ and $|E_2(M_0^*,\varphi_0)| \ge 1$.
Recall that	for each $e\in M_0^*$, $e$  is fully $G$-saturated in  $H_e+e$.
Thus we have the  following observation: for an edge $f_{uv}\in M$ with $V(f_{uv})=\{u,v\}$, if $\{u,v\} \cap V(H_e)=\emptyset$ for any $e\in M_0^*$, then $f_{uv}$ cannot be a T2-improper edge.

Since $|E_2(M_0^*,\varphi_0)| \ge 1$, we consider one T2-improper edge in $M$, say $f_{uv}$ with $V(f_{uv})=\{u,v\}$.
Suppose  that $f_{uv}$ is T2-improper at $u$ and $\Phi(f_{uv})=i\in [k]$ (as $\varphi_0$ is a $k$-edge-coloring, $i\ne k+1=\Delta+\mu$). Then there exist
$e_{xy} \in E_G(x,y)\cap M^*_0$ and a maximal $k$-dense subgraph $H$ of $G-(M\cup M^*_0)$ such that $V(e_{xy}) \subseteq V(H)$ and $f_{uv}$ and $e_{xy}$ are both adjacent  in $G$ to an $i$-edge $e_{yu} \in E_{H}(y,u)$.  Since $M$ is a distance-$3$ matching and $\diam(H)\le 2$,  we have $ V(H)\cap V(M\setminus\{f_{uv}\})=\emptyset$.
We will modify $\varphi_0$
into a new coloring such that $f_{uv}$ is not  T2-improper at $u$ under this new coloring  and that no other edge of $M_0^*$ is changed into a new T2-improper edge.
We  consider the  three  cases below regarding the location of $f_{uv}$ with respect to $H$.

	\noindent{\bf Case 1:}  $f_{uv}$ is not improper at $v$, or $f_{uv}$ is T1-improper  at $v$ but $v\notin V(H)$.

Let $F_x$ be a maximal multi-fan  at $x$ with respect to $e_{xy}$ and $(\varphi_0)_{H}$ in $H+e_{xy}$.
There exist at least one $\Delta$-vertex in $V(F_x)\setminus\{x,y\}$ by Lemma \ref{MD}$(a)$ and a linear sequence at $x$ from $y$ to this $\Delta$-vertex in $F_x$.
We consider two subcases as follows.

	\begin{figure}[!ht]
		\begin{center}
			\centering
			\scalebox{0.32}{\includegraphics{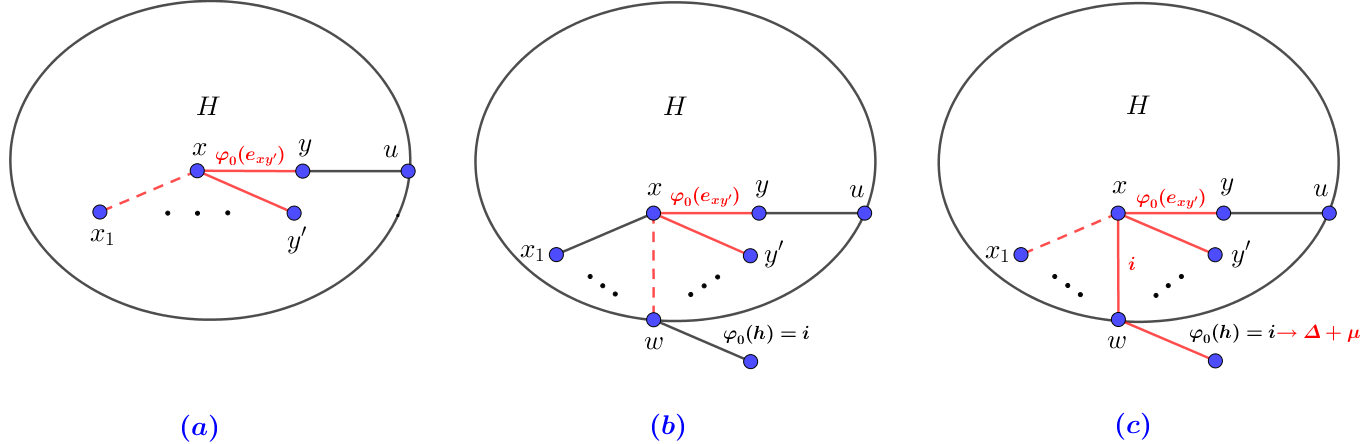}}
			\caption{Operations I, II and III in Case $1$. (The edges of the dashed line represent uncolored edges.)}
		\end{center}
	\end{figure}

	{\bf Subcase 1.1:}	$V(F_x)\setminus\{x,y\}$ has a $\Delta$-vertex $x_1$ and there is a
 linear sequence	$S$ at $x$ from $y$ to $x_1$ such that $S$ contains no $i$-edge or $S$ contains no vertex $w$ such that $w$ is incident with an $i$-edge of  $\partial_{G-(M\cup M_0^*)}(H)$.

Let $S=(y, e_{xy'}, y', \ldots, e_{xx_1}, x_1)$ be the linear sequence (where $y'=x_1$ is possible).
We apply Operation I as
follows: apply  a shifting in $S$ from $y$ to  $x_1$, color $e_{xy}$ with $\varphi_0(e_{xy'})$, uncolor $e_{xx_1}$, and replace $e_{xy}$ by $e_{xx_1}$ in $M^*_0$. See Figure $1(a)$.
Since $x_1$ is not incident with any edge in $M\cup M^*_0$, $M^*_1:=(M^*_0\backslash \{e_{xy}\})\cup \{e_{xx_1}\}$ is  a matching.  Denote  $H_1:=H+e_{xy}-e_{xx_1}$. Let $\psi$ be the $k$-edge coloring of  $H_1$    after Operation I.
Note that for any vertex $z\in V(H_1)$ that is incident with an edge of $\partial_{G-(M\cup M_1^*)}(H_1)$, if $\overline{\psi}(z) \ne \overline{(\varphi_0)}_{H}(z)$, then $z\in V(S)$.
By the condition of Subcase 1.1 and Operation I, there is no such vertex $w$ such that
$w$ is incident with both an $i$-edge of $E(S)$ and an $i$-edge of $\partial_{G-(M\cup M_1^*)}(H_1)$. Thus we can rename some color classes of $\psi$ but keep the color $i$ unchanged to match all  colors on   edges of $\partial_{G-(M\cup M_1^*)}(H_1)$. In this way we obtain a (proper) $k$-edge-coloring $\varphi_1$ of $G-(M\cup M^*_1)$
by Lemma \ref{lem-consisting}$(b)$.

We claim that $(M^*_1,\emptyset,\varphi_1)$  is a prefeasible triple.
As $M_{\Delta+\mu} \cup M_1^*$ is a matching, we verify that $M_1^*$ and $\varphi_1$ satisfy the corresponding conditions.
Clearly  $H_1$ is $k$-dense with  $V(H_1)=V(H)$ and $\partial_{G-(M\cup M_1^*)}(H_1)=\partial_{G-(M\cup M_0^*)}(H)$ and $\chi'(H_1)=\chi'(H)=k$, and $e_{xx_1}$
is $k$-critical and fully $G$-saturated in $H_1+e_{xx_1}$.   Furthermore, as distinct maximal $k$-dense subgraphs are vertex-disjoint
we know that each edge  $e\in M_1^*\setminus \{e_{xx_1}\}$
is still contained in a $k$-dense subgraph of $G-(M\cup M_1^*)$ such that  $e$ is $k$-critical and fully $G$-saturated in the graph $H_e+e$ if $e$ is adjacent in $G$ to an edge of $E_2(M^*_1,\varphi_1)$, where $H_e$ is the unique maximal $k$-dense subgraph of $G-(M\cup M_0^*)$ containing $V(e)$ if $H_e$ and $H_1$ are vertex-disjoint, and $H_e=H_1$ otherwise.
Since $\varphi_1$ is a $k$-edge-coloring of $G-(M\cup M_1^*)$, $H_e$ is strongly $\varphi_1$-closed  for each $e\in M_1^*$.  Therefore, $(M^*_1,\emptyset,\varphi_1)$  is a prefeasible triple.

Next, we claim that $|E_2(M_1^*,\varphi_1)|= |E_2(M_0^*,\varphi_0)|-1$.
Note that under $\varphi_1$, we still have $\varphi_1(e_{yu})=i$. Since $e_{xy}, e_{yu} \in E(H_1)$, $e_{xx_1}\in M^*_1$ and $e_{xx_1}$
is not adjacent to $e_{yu}$ in $G-(M\cup M_1^*)$,
we see that now $f_{uv}$ is no longer  T2-improper at $u$  but    T1-improper  at $u$ with respect to $M_1^*$ and $\varphi_1$.
For any edge $f\in M\setminus \{f_{uv}\}$, since
both $x$ and $x_1$ are $\Delta$-vertices of $H+e_{xy}$ and $V(H_1)\cap V(M\setminus\{f_{uv}\})=\emptyset$,
we see that the distance between $f$ and $e_{xx_1}$ in $G-(M\cup M_1^*)$ is at least $2$.  Thus the property of $f$ being T1-improper or
T2-improper is not changed under $M_1^*$ and $\varphi_1$. Thus
the new triple $(M^*_1,\emptyset,\varphi_1)$  is  prefeasible with $|E_1(M_1^*,\varphi_1)|=|E_1(M_0^*,\varphi_0)|+1$ and $|E_2(M_1^*,\varphi_1)|= |E_2(M_0^*,\varphi_0)|-1$, and so we can consider $(M^*_1,\emptyset,\varphi_1)$ instead.

 {\bf Subcase 1.2:}	
For any $\Delta$-vertex in $V(F_x)\setminus\{x,y\}$,  any linear sequence from $y$ to this $\Delta$-vertex contains both  an $i$-edge $h_i$ and a vertex $w$ such that $w$ is incident with an $i$-edge  $h$ of  $\partial_{G-(M\cup M_0^*)}(H)$.

Let $F\subseteq F_x$ be the maximal multi-fan at $x$ without any $i$-edge with respect to $e_{xy}$ and $(\varphi_0)_{H}$. By the condition of Subcase 1.2, $F$ does not contain any $\Delta$-vertex from $V(F)\backslash \{x,y\}$ in $H$. By Lemma \ref{MD}$(c)$, there exists a vertex $z^*\in V(F)\backslash \{x,y\}$ with $i\in \overline{(\varphi_0)}_{H}(z^*)$ and $d_{H}(z^*)=\Delta-1$. Since $V(F_x)$ is $(\varphi_0)_{H}$-elementary by Lemma \ref{MF}$(a)$ and every color on edges of $\partial_{G-(M\cup M_0^*)}(H)$ under $\varphi_0$ is a missing color at some vertex of $H$ under $(\varphi_0)_{H}$, it follows that $z^*=w$, i.e., $d_{H}(w)=\Delta-1$ and $d_{G-(M\cup M^*_0)}(w)=\Delta$.
Thus the $i$-edge $h$ is the only edge incident with $w$ from $\partial_{G-(M\cup M_0^*)}(H)$, and $w$ is not adjacent in $G$ to any edge from $M\cup M^*_0$. Let $S=(y, e_{xy'}, y', \ldots, e_{xx_1}, x_1)$ be a linear sequence at $x$ from $y$ to $x_1$.
Notice that  $w$ is in $S$ by the condition of Subcase 1.2. We consider the following two subcases  according whether the boundary $i$-edge $h$ belongs to $E_1(M_0^*,\varphi_0)$.
	
		{\bf Subcase 1.2.1:}		
	$h\notin E_1(M_0^*,\varphi_0)$, i.e., $h$ is not adjacent in $G$ to any precolored $i$-edge in $M$.
	
Let $e_{xw}\in E_{H}(x,w)$ be an edge in $S$.
We apply Operation II  as follows:
apply  a shifting in $S$ from $y$ to $w$, color $e_{xy}$ with $\varphi_0(e_{xy'})$,
uncolor $e_{xw}$, and replace $e_{xy}$ by $e_{xw}$ in $M^*_0$. See Figure $1(b)$.
Since $d_{G-(M\cup M^*_0)}(w)=\Delta$, $M^*_1:=(M^*_0\backslash \{e_{xy}\})\cup \{e_{xw}\}$ is a matching. Denote $H_1:=H+e_{xy}-e_{xw}$.
Let $\psi$ be the $k$-edge coloring of $H_1$  after Operation II. Note that for any vertex $z\in V(H_1)$ that is incident with an edge of $\partial_{G-(M\cup M_1^*)}(H_1)$, if $\overline{\psi}(z) \ne \overline{(\varphi_0)}_{H}(z)$, then $z$ is contained in the subsequence of $S$ from $y$ to $w$.
Since $h$ is the only $i$-edge of $\partial_{G-(M\cup M_1^*)}(H_1)$, there is no such vertex $w$ such that
$w$ is incident with both an $i$-edge contained in the subsequence of $S$ from $y$ to $w$  and an $i$-edge of $\partial_{G-(M\cup M_1^*)}(H_1)$ after Operation II.
Thus we can rename some color classes of $\psi$ but keep the color $i$ unchanged to match all colors on  boundary edges of $\partial_{G-(M\cup M_1^*)}(H_1)$. In this way we obtain a (proper) $k$-edge-coloring $\varphi_1$ of $G-(M\cup M^*_1)$
by Lemma \ref{lem-consisting}$(b)$.

By the similar  argument in the proof of Subcase 1.1,    it can be verified that $(M^*_1,\emptyset,\varphi_1)$ is prefeasible,
and that  $f_{uv}$ is no longer  T2-improper  at $u$  but  T1-improper  at $u$ with respect to $M_1^*$ and $\varphi_1$.
For any edge $f\in M\setminus \{f_{uv}\}$,  we see that the distance between  $f$ and $e_{xw}$  is at least 2 or just 1 when $h$ is adjacent  in $G$ to $f$ with $\Phi(f)\neq i$.  Thus the property of $f$ being T1-improper or
T2-improper is not changed under $M_1^*$ and $\varphi_1$. Thus
the new triple $(M^*_1,\emptyset,\varphi_1)$  is  prefeasible with $|E_1(M_1^*,\varphi_1)|=|E_1(M_0^*,\varphi_0)|+1$ and $|E_2(M_1^*,\varphi_1)|= |E_2(M_0^*,\varphi_0)|-1$, and so we can consider $(M^*_1,\emptyset,\varphi_1)$ instead.

	{\bf Subcase 1.2.2:}	
	$h\in E_1(M_0^*,\varphi_0)$, i.e.,	$h$ is adjacent in $G$ to some precolored $i$-edge $f_i$ in $M$.
	
We apply Operation III as follows:   recolor the $i$-edge $h$ with the color $\Delta+\mu$,
apply  a shifting in $S$ from $y$ to $x_1$, color $e_{xy}$ with $\varphi_0(e_{xy'})$,
uncolor  $e_{xx_1}$, and replace $e_{xy}$ by $e_{xx_1}$ in $M^*_0$. See Figure $1(c)$.
By  the same argument as  in the proof of Subcase 1.1,  we know that $M^*_1:=(M^*_0\backslash \{e_{xy}\})\cup \{e_{xx_1}\}$ is a matching. Denote $H_1:=H+e_{xy}-e_{xx_1}$.
Let $\psi$ be the $k$-edge coloring of $H_1$  after Operation III.
Note that there is no $i$-edge in $\partial_{G-(M\cup M_1^*)}(H_1)$ after Operation III.  By the similar argument as in the proof of Subcase 1.1,  we can rename some color classes of $\psi$ but keep the color $i$ unchanged to match all  colors on  edges of $\partial_{G-(M\cup M_1^*)}(H_1)$. In this way we obtain a (proper) $(k+1)$-edge-coloring $\varphi_1$ of $G-(M\cup M^*_1)$
by Lemma \ref{lem-consisting}$(b)$.

We claim that $(M^*_1,\emptyset,\varphi_1)$  is a prefeasible triple.  As $M\cup M_1^*$
is a matching and $h$ is adjacent to $f_i$ and $\Phi(f_i)=i\in [k]$,  it follows that
$h$ is not adjacent to any edge from $M_{\Delta+\mu}\cup M_1^*$, which implies that $M_{\Delta+\mu}\cup M_1^*\cup \{h\}$
is a matching.
By the same argument as in the proof of Subcase 1.1,  we know that $e_{xx_1}$
is $k$-critical and fully $G$-saturated in $H_1+e_{xx_1}$, and each edge  $e\in M_1^*\setminus \{e_{xx_1}\}$
is still contained in a $k$-dense subgraph of $G-(M\cup M_1^*)$ such that  $e$ is $k$-critical and fully $G$-saturated in the graph $H_e+e$ if $e$ is adjacent in $G$ to an edge of $E_2(M^*_1,\varphi_1)$, where $H_e$ is the unique maximal $k$-dense subgraph of $G-(M\cup M_0^*)$ containing $V(e)$ if $H_e$ and $H_1$ are vertex-disjoint, and $H_e=H_1$ otherwise.
If the color $\Delta+\mu$  is not used on  edges of $\partial_{G-(M\cup M_1^*)}(H_e)$,
then colors on edges of $\partial_{G-(M\cup M_1^*)}(H_e)$ are all distinct by the fact that
$H_e$ is strongly $\varphi_1$-closed.
If the  color $\Delta+\mu$  is  used on  edges of $\partial_{G-(M\cup M_1^*)}(H_e)$, then it was used on exactly one edge of $\partial_{G-(M\cup M_1^*)}(H_e)$. This, together with the fact that  $H_e$ is $(\varphi_1)_{H_e}$-elementary, implies that colors on edges of $\partial_{G-(M\cup M_1^*)}(H_e)$ are all distinct.
Therefore, $(M^*_1,\emptyset,\varphi_1)$  is a prefeasible triple.

By the same argument as in the proof of Subcase 1.1,
we know that now $f_{uv}$ is no longer  T2-improper at $u$  but   T1-improper  at $u$ with respect to $M_1^*$ and $\varphi_1$, and that for any edge $f\in M\setminus \{f_{uv}\}$,
the distance between  $f$ and $e_{xx_1}$ in $G-(M\cup M_1^*)$ is at least 2.   Except the $i$-edge $f_i$ of $M$
that is adjacent  in $G$ to $h$,
the property of $f$ being T1-improper or
T2-improper is not changed under $M_1^*$ and $\varphi_1$.  The edge $f_i$
is originally T1-improper at $w_i$, and now is no longer improper at $w_i$
with respect to $\varphi_1$, where we assume  $h\in E_G(w,w_i)$.
Thus   $|E_1(M_1^*,\varphi_1)|=|E_1(M_0^*,\varphi_0)|+1-1$ and $|E_2(M_1^*,\varphi_1)|= |E_2(M_0^*,\varphi_0)|-1$, and so we can consider $(M^*_1,\{h\},\varphi_1)$ instead.
Note that  assigning  the color $\Delta+\mu$ to $h$ will not affect the modification of $\varphi_0$ into $\varphi$ and $M_0^*$ into $M^*$,  since  $h\in E_1(M_0^*,\varphi_0)$ and we will assign the color  $\Delta+\mu$ to all edges in $E_1(M_0^*,\varphi_0)$ in the final process.

	\noindent{\bf Case 2:}	
	$f_{uv}$ is T2-improper at $v$ with $v\in V(H')$ for a maximal $k$-dense subgraph $H'$ other than $H$.
	
For this case, we apply the same operations as we did in Case 1 first with respect to the vertex $u$ in $H$ and then with respect to the vertex $v$ in $H'$.
Recall that $V(H)\cap V(H')=\emptyset$
and $E_1(M_0^*,\varphi_0)$ is a matching.  By Case $1$, the operations applied within $G[V(H)]$ or $G[V(H)]+h_u$ do not affect the
operations applied within $G[V(H')]$ or $G[V(H')]+h_v$, where $h_u$ and $h_v$ are the two possible $i$-edges with $h_u\in\partial_{G-(M\cup M_0^*)}(H)\cap E_1(M_0^*,\varphi_0)$ and $h_v\in\partial_{G-(M\cup M_0^*)}(H')\cap  E_1(M_0^*,\varphi_0)$. Furthermore,  if $h_u$ and $h_v$ exist at the same time, then $V(h_u)\cap V(h_v) =\emptyset$ and there is no maximal $k$-dense subgraph $H''$ other than $H$ and $H'$ such that $V(H'')\cap V(h_u)\neq\emptyset$ and  $V(H'')\cap V(h_v)\neq\emptyset$. Denote the matching resulting from $M_0^*$
by $M_1^*$, and the  coloring resulting from $\varphi_0$ by $\varphi_1$.
By Case $1$, $E_{M_1^*}^{\varphi_1}\subseteq\{h_u,h_v\}$, $M_{\Delta+\mu}\cup M_1^*\cup \{h_u,h_v\}$ is a matching,
and $(M^*_1, E_{M_1^*}^{\varphi_1},\varphi_1)$ also satisfies Conditions $(b)$ and $(c)$.  Thus  $(M^*_1, E_{M_1^*}^{\varphi_1},\varphi_1)$ is a prefeasible triple.
With respect to $M_1^*$
and $\varphi_1$, $f_{uv}$ is no longer T2-improper but is T1-improper at both $u$ and $v$. Furthermore, we have  $|E_1(M_1^*,\varphi_1)| \ge |E_1(M_0^*,\varphi_0)|$ and $|E_2(M_1^*,\varphi_1)|= |E_2(M_0^*,\varphi_0)|-2$. Thus we can consider $(M^*_1, E_{M_1^*}^{\varphi_1},\varphi_1)$ instead.

	\noindent{\bf Case 3:}	
$f_{uv}$ is T1-improper or T2-improper at $v$ with $ v\in V(H)$.

Assume first that $d_{H}(b)<\Delta$. Let $e_{bv}\in E_{H}(b,v)$ with $\varphi_0(e_{bv})=i$.
If $f_{uv}$ is T1-improper at $v$, then we apply the same operations with respect to $u$ as we did in Case $1$.
Denote the new matching resulting from $M_0^*$ by $M_1^*$, and the new  coloring resulting from $\varphi_0$ by $\varphi_1$. Then the vertex $b$ is not incident in $G$ with any edge of $M_1^*$ by Operations I-III in Case $1$. Thus $f_{uv}$ is no longer  T2-improper at $u$ but T1-improper at $u$ with respect to $M_1^*$ and $\varphi_1$. Furthermore, we have $|E_1(M_1^*,\varphi_1)| \ge  |E_1(M_0^*,\varphi_0)|$ and $|E_2(M_1^*,\varphi_1)|= |E_2(M_0^*,\varphi_0)|-1$. Thus we can consider $(M^*_1, E_{M_1^*}^{\varphi_1},\varphi_1)$ instead.

If $f_{uv}$ is T2-improper at $v$,
let $e_{ab}\in M_0^*$ with $V(e_{ab})=\{a,b\}$.
We apply the same operations with respect to $u$ as we did in Case $1$.
Denote the resulting matching  by $M_1^*$, and the resulting coloring by $\varphi_1$.  With respect to $M_1^*$
and $\varphi_1$, the edge $f_{uv}$ is still T2-improper at $v$ as $d_{H}(a)<\Delta$ and $d_{H}(b)<\Delta$. By Case $1$, now $f_{uv}$ is no longer T2-improper at $u$ but T1-improper at $u$ with respect to the prefeasible triple $(M^*_1, E_{M_1^*}^{\varphi_1},\varphi_1)$,
where $E_{M_1^*}^{\varphi_1}=\emptyset$ or $\{h\}$ with some vertex $w$ and its incident $i$-edge $h\in\partial_{G-(M\cup M_0^*)}(H)\cap E_1(M_0^*,\varphi_0)$.
Denote by $H_1$  the new $k$-dense subgraph after the operations with respect to $u$ in $H+e_{xy}$.	 In particular, the situation under $(M^*_1,\emptyset,\varphi_1)$ is actually the
same as the case $d_{H}(b)=\Delta$ in the previous paragraph since  now  $d_{H_1}(y)=\Delta$.

Thus we consider only the case that $f_{uv}$ is T2-improper at $v$, T1-improper at $u$ and $d_{H_1}(y)=\Delta$. Consider a maximal multi-fan $F_a$ at $a$ with respect to $e_{ab}$ and $(\varphi_1)_{H_1}$ in $H_1+e_{ab}$.  Clearly we can apply the same operations in Case $1$ for $v$ so that $f_{uv}$ is no longer  T2-improper  at $v$ with respect to the resulting matching $M_2^*$ and coloring $\varphi_2$, unless these operations would have to put one edge $e_{ay}\in E_{H_1}(a,y)$ into $M^*_2$. Then $f_{uv}$ would become T2-improper at $u$ again with respect to $M_2^*$ and $\varphi_2$.
The only operations that have to uncolor an edge of $H_1$ incident with $y$
are Operations I and III.
Therefore, we make the following two assumptions on $F_a$
in the rest of our proof.
		 \begin{enumerate}[(1)]
		 	\item $y$ is the only $\Delta$-vertex in $V(F_a)\backslash\{a,b\}$.
		 	\item  If a linear sequence in $F_a$ at $a$ from $b$ to $y$ contains a vertex $w'$ such that $d_{H_1}(w')=\Delta-1$ and $w'$ is incident with an  $i$-edge   $ h'\in \partial_{G-(M\cup M_1^*)}(H_1)$,  then $h'\in E_1(M_1^*, \varphi_1)$.
		 \end{enumerate}

Let $F_b$ be a maximal multi-fan at $b$ with respect to $e_{ab}$ and $(\varphi_1)_{H_1}$ in $H_1+e_{ab}$. We consider the following  three subcases.

	{\bf Subcase 3.1:} $F_{b}$ contains a linear sequence $S$ at $b$ from $a$ to $y$  such that $S$ does not contain any $i$-edge.
	
Let $S=(a, e_{ba'}, a',\ldots, e_{by},y)$ be the linear sequence (where $a'=y$ is possible).
We apply a shifting in $S$ from $a$ to $y$, color $e_{ab}$ with $\varphi_1(e_{ba'})$,
uncolor $e_{by}$. See Figure $2$(a)-(b). Note that $M_2^*:=(M_1^*\setminus \{e_{ab}\}) \cup \{e_{by}\}$ is a matching,
and $H_2:=H_1+e_{ab}-e_{by}$ is a $k$-dense subgraph of $G-(M\cup M_2^*)$.  As $S$ does not contain any $i$-edge, by Lemma \ref{lem-consisting}$(b)$, we obtain a $k$-edge-coloring $\varphi_2$ of $G-(M\cup M_2^*)$.
Note that  $f_{uv}$ is  T2-improper at both $u$ and $v$ with respect to $M_2^*$
and $\varphi_2$. However,  we have $\Phi(f_{uv})=i$,  $\varphi_2(e_{bv})=\varphi_2(e_{yu})=i$, and $e_{by} \in M_2^*$ ($bvuyb$ is a cycle with length $4$ in $G$).
By assigning the color $i$ to $e_{by}$ and recoloring  $e_{bv}$ and $e_{yu}$ with the color $\Delta+\mu$, we obtain a new matching $M^*_3:=M^*_2\backslash \{e_{by}\}=M^*_1\backslash \{e_{ab}\}$  of $G-V(M)$ and a new $(k+1)$-edge-coloring $\varphi_3$ of $G-(M\cup M^*_3)$. See Figure $2(c)$.
The edge $f_{uv}$ is now not improper at neither of its endvertices.
Note that $E_{M_3^*}^{\varphi_3}= \{e_{bv},e_{yu}\}$ if $E_{M_1^*}^{\varphi_1}=\emptyset$ and $E_{M_3^*}^{\varphi_3}=\{h,e_{bv},e_{yu}\}$ if $E_{M_1^*}^{\varphi_1}= \{h\} $.
Since $E_{M_3^*}^{\varphi_3}\subseteq (E_1(M_0^*,\varphi_0) \cup E_2(M_0^*,\varphi_0))$ is a matching, and those edges in $E_{M_3^*}^{\varphi_3}$ do not share any endvertex
with edges in $M_{\Delta+\mu}\cup M_3^*$,
it follows that  $M_{\Delta+\mu}\cup M_3^*\cup E_{M_3^*}^{\varphi_3}$ is a matching. Note that $V(H_2)\cap V(M\setminus\{f_{uv}\})=\emptyset$. For each  $e\in M_3^*$ such that $e$ is adjacent in $G$ to an edge of $E_2(M^*_3,\varphi_3)$, $e$ is still $k$-critical and fully $G$-saturated in the graph $H_e+e$, where $H_e$ is still the unique maximal $k$-dense subgraph of $G-(M\cup M^*_0)$ containing $V(e)$ and $H_e$ is also strongly $\varphi_3$-closed.  Thus the new triple  $(M^*_3,E_{M_3^*}^{\varphi_3},\varphi_3)$ is prefeasible.
Furthermore,
 $|E_1(M_3^*,\varphi_3)|  =  |E_1(M_1^*,\varphi_1)|-1\ge |E_1(M_0^*,\varphi_0)|-1$ and $|E_2(M_3^*,\varphi_3)|= |E_2(M_1^*,\varphi_1)|-1 =|E_2(M_0^*,\varphi_0)|-2$. Thus we can consider $(M^*_3, E_{M_3^*}^{\varphi_3},\varphi_3)$ instead.

	\begin{figure}[!ht]
		\begin{center}
			\centering
			\scalebox{0.32}{\includegraphics{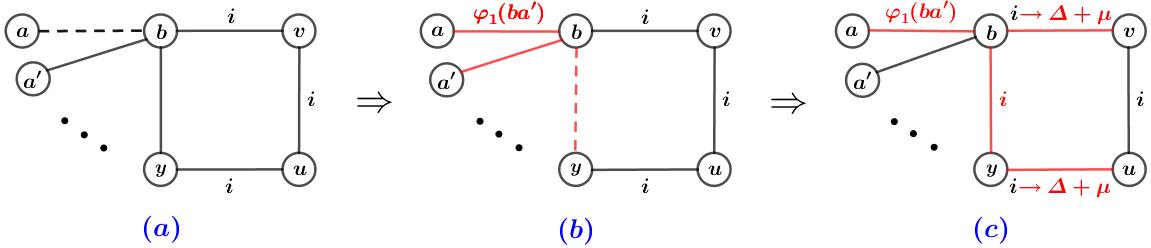}}
			\caption{Operation in Subcase $3.1$. (The edges of the dashed line represent uncolored edges.)}
		\end{center}
	\end{figure}

	{\bf Subcase 3.2:} $F_{b}$ contains a vertex $w''$ with $d_{H_1}(w'')=\Delta-1$ and $i\in\overline{(\varphi_1)}_{H_1}(w'')$.

The $i$-edge $e_{bv}$ is in $F_b$ by the maximality of $F_b$.
Let $S=(a,e_{ba'}, a',\ldots, e_{bw''}, w'', e_{bv},v)$  be a linear sequence  at $b$ from $a$ to $v$ in $F_b$ (where $a=a'$ and $a'=w''$   are possible).
Since $i\in\overline{(\varphi_1)}_{H_1}(w'')$, we have  that either
$i\in\overline{\varphi}_1(w'')$
or $w''$ is incident with an $i$-edge $h''\in\partial_{G-(M\cup M_1^*)}(H_1)$.

Assume first that $i\in\overline{\varphi}_1(w'')$ or $w''$ is incident with an $i$-edge $h''\in\partial_{G-(M\cup M_1^*)}(H_1)$ such that
$h''\in E_1(M_1^*,\varphi_1)$. We apply a shifting in $S$ from $a$ to $v$,
color  $e_{ab}$ with $\varphi_1(e_{ba'})$,  and uncolor  $e_{bv}$.
Note that  $e_{bw''}$ was recolored by the color $i$ in the shifting operation.
We then recolor   the  $i$-edge $h''$ with the color $\Delta+\mu$ if $h''$ exists,
and rename  some color classes of $H_2:=H_1+e_{ab}-e_{bv}$ but keep the color $i$ unchanged without producing any improper $i$-edge by Lemma \ref{lem-consisting}$(b)$. Finally we assign the color $\Delta+\mu$ to $e_{bv}$. Note that $h\neq h''$ since $\varphi_1(h)=\Delta+\mu\neq i=\varphi_1(h'')$, and $h$ and $h''$ cannot both exist in  $\partial_{G-(M\cup M_0^*)}(H)=\partial_{G-(M\cup M_1^*)}(H_1)$ since otherwise $\varphi_0(h)=\varphi_0(h'')=i$ contradicting that $H$ is strongly $\varphi_0$-closed.
Now we obtain a new matching $M^*_2:=M^*_1\backslash \{e_{ab}\}$ of $G-V(M)$ and a new (proper) $(k+1)$-edge-coloring $\varphi_2$ of $G-(M\cup M^*_2)$  such that $f_{uv}$ is no longer  T2-improper at $v$ or even  T1-improper at $v$ with respect to a new triple $(M^*_2,E_{M_2^*}^{\varphi_2},\varphi_2)$, where $E_{M_2^*}^{\varphi_2}= \{e_{bv}\}$  if $E_{M_1^*}^{\varphi_1}=\emptyset$ but $h''$ does not exist, $E_{M_2^*}^{\varphi_2}=\{e_{bv},h''\}$  if $E_{M_1^*}^{\varphi_1}=\emptyset$ and $h''$ exists, and $E_{M_2^*}^{\varphi_2}=\{e_{bv},h\}$  if $E_{M_1^*}^{\varphi_1}=\{h\}$.
Since $E_{M_2^*}^{\varphi_2}\subseteq (E_1(M_0^*,\varphi_0) \cup E_2(M_0^*,\varphi_0))$ is a matching, and those edges in $E_{M_2^*}^{\varphi_2}$ do not share any endvertex
with edges in $M_{\Delta+\mu}\cup M_2^*$,
it follows that $M_{\Delta+\mu}\cup M_2^*\cup E_{M_2^*}^{\varphi_2}$ is a matching. Note that $V(H_2)\cap V(M\setminus\{f_{uv}\})=\emptyset$. By the similar argument as in the proof of Subcase $3.1$,  the new triple  $(M^*_2,E_{M_2^*}^{\varphi_2},\varphi_2)$ is  prefeasible.
Furthermore, $|E_1(M_2^*,\varphi_2)|  \ge |E_1(M_0^*,\varphi_0)|$ and $|E_2(M_2^*,\varphi_2)|=|E_2(M_0^*,\varphi_0)|-2$. Thus we can consider $(M^*_2, E_{M_2^*}^{\varphi_2},\varphi_2)$ instead.

Now we may assume that the $i$-edge $h'' \not\in E_1(M_1^*,\varphi_1)$.
Since $h$ and $h''$ cannot both exist, we have $E_{M_1^*}^{\varphi_1}=\emptyset$.
Note that the vertex $w''\notin V(F_a)$ by Assumption $(2)$ prior to Subcase $3.1$. Moreover, $w''$ is not incident with any edge in $M\cup M^*_1$ and $w''$ is only incident with the $i$-edge $h''$ in $\partial_{G-(M\cup M_1^*)}(H_1)$. Since $d_{G-(M\cup M^*_1)}(w'')=\Delta$ and $\varphi_1$ is a $k$-edge-coloring of $G-(M\cup M^*_1)$ with $k\geq \Delta+1$, there exists a color $\alpha\in\overline{\varphi}_1(w'')$ with $\alpha\neq i$. Since $V(H_1)$ is $(\varphi_1)_{H_1}$-elementary, there exists an $\alpha$-edge $e_1$ incident with the vertex $a$. Thus we can define a maximal multi-fan at $a$, denoted by  $F'_a$,  with respect to $e_1$ and $(\varphi_1)_{H_1}$ in $H_1+e_1$.
(Notice that $e_1$ is colored by the color $\alpha$ in $F'_a$.) Moreover, $V(F'_a)$ is $(\varphi_1)_{H_1}$-elementary since $V(H_1)$ is $(\varphi_1)_{H_1}$-elementary. By Lemma \ref{MD}$(b)$ and Assumption $(1)$ prior to Subcase $3.1$, we have $e_{F_a}(a,b')=e_{H_1+e_{ab}}(a,b')=\mu$ for any vertex $b'$ in $V(F_a)\backslash\{a\}$. Therefore, $V(F'_a)\backslash\{a\}$ and $V(F_a)\backslash\{a\}$ are disjoint, since otherwise we have $V(F'_a)\subseteq V(F_a)$ and $\alpha\in \overline{(\varphi_1)}_{H_1}(b')$ for some $b'\in V(F_a)$ implying $b'=w''\in V(F_a)$, a contradiction. Note that if $w''\notin V(F'_a)$, then $V(F'_a)\backslash \{a\}$ must contain a $\Delta$-vertex in $H_1$, since otherwise Lemma \ref{MD}$(d)$ and the fact $(\varphi_1)_{H_1}(e_1)=\alpha\in\overline{\varphi}_1(w'')$ imply that $w''\in V(F'_a)$, a contradiction. Thus $F'_a$ contains a linear sequence $S'=(b_1,e_2,b_2,\ldots, e_t,b_t)$ at $a$, where $b_1\in V(e_1)$,  $b_t$ (with $t\ge 1$) is a $\Delta$-vertex if $w''\notin V(F'_a)$, and $b_t$ is $w''$ if $w''\in V(F'_a)$.
Notice that $b_t$ is not incident with any edge in $M\cup M^*_1$ by our choice of $b_t$.  Moreover, $b_t\neq y$ since $V(F'_a)\backslash\{a\}$ and $V(F_a)\backslash\{a\}$ are disjoint. Let $\beta$ ($\beta\neq i$) be a color in $\overline{\varphi}_1(b)$. By Lemma \ref{MF}$(b)$, we have $P_b(\beta,\alpha)=P_{w''}(\beta,\alpha)$. We then consider the following two subcases according the set $(V(S')\backslash\{a\})\cap (V(S)\backslash\{a\})$.

We first assume that $(V(S')\backslash\{a\})\cap (V(S)\backslash\{a\})\subseteq\{b_t\}$. If $e_1\notin P_b(\beta,\alpha)$, then we  apply  a Kempe change on $P_{[b,w'']}(\beta,\alpha)$, uncolor $e_1$ and color $e_{ab}$ with $\alpha$.
If  $e_1\in P_b(\beta,\alpha)$ and $P_b(\beta,\alpha)$ meets $b_1$ before $a$, then we  apply a Kempe change on $P_{[b,b_1]}(\beta,\alpha)$, uncolor $e_1$ and color $e_{ab}$ with $\alpha$. If  $e_1\in P_b(\beta,\alpha)$ and $P_{w''}(\beta,\alpha)$ meets $b_1$ before $a$, then we uncolor $e_1$, apply a Kempe change on $P_{[w'',b_1]}(\beta,\alpha)$,  apply a shifting in $S$ from $a$ to $w''$, color  $e_{ab}$ with $\varphi_1(e_{ba'})$, and recolor $e_{bw''}$ with $\beta$. In all three cases above,  $e_{ab}$ is colored with a color in $[k]$ and $e_1$ is uncolored. Finally we apply a shifting in $S'$ from $b_1$ to $b_t$, color $e_1$ with $\varphi_1(e_2)$, and uncolor $e_t$. Notice that the above shifting in $S'$ does nothing if $t=1$. Denote $H_2:=H_1+e_{ab}-e_t$. Since $H_2$ is also $k$-dense and $\chi'(H_2)=k$, we can rename some color classes of $E(H_2)$ but keep the color $i$ unchanged to match all colors on boundary edges without producing any improper $i$-edge by Lemma \ref{lem-consisting}$(b)$.
Now we obtain a new matching $M^*_2:=(M^*_1\backslash \{e_{ab}\})\cup \{e_t\}$  and
a new (proper) $k$-edge-coloring $\varphi_2$ of $G-(M\cup M^*_2)$ such that $f_{uv}$ is no longer T2-improper at $v$ but T1-improper at $v$ with respect to the new prefeasible triple $(M^*_2,\emptyset,\varphi_2)$.
Furthermore, $|E_1(M_2^*,\varphi_2)|=|E_1(M_0^*,\varphi_0)|+2$ and $|E_2(M_2^*,\varphi_2)|=|E_2(M_0^*,\varphi_0)|-2$.	Thus we can consider $(M^*_2,\emptyset,\varphi_2)$ instead.

	\begin{figure}[!ht]
		\begin{center}
			\centering
			\scalebox{0.34}{\includegraphics{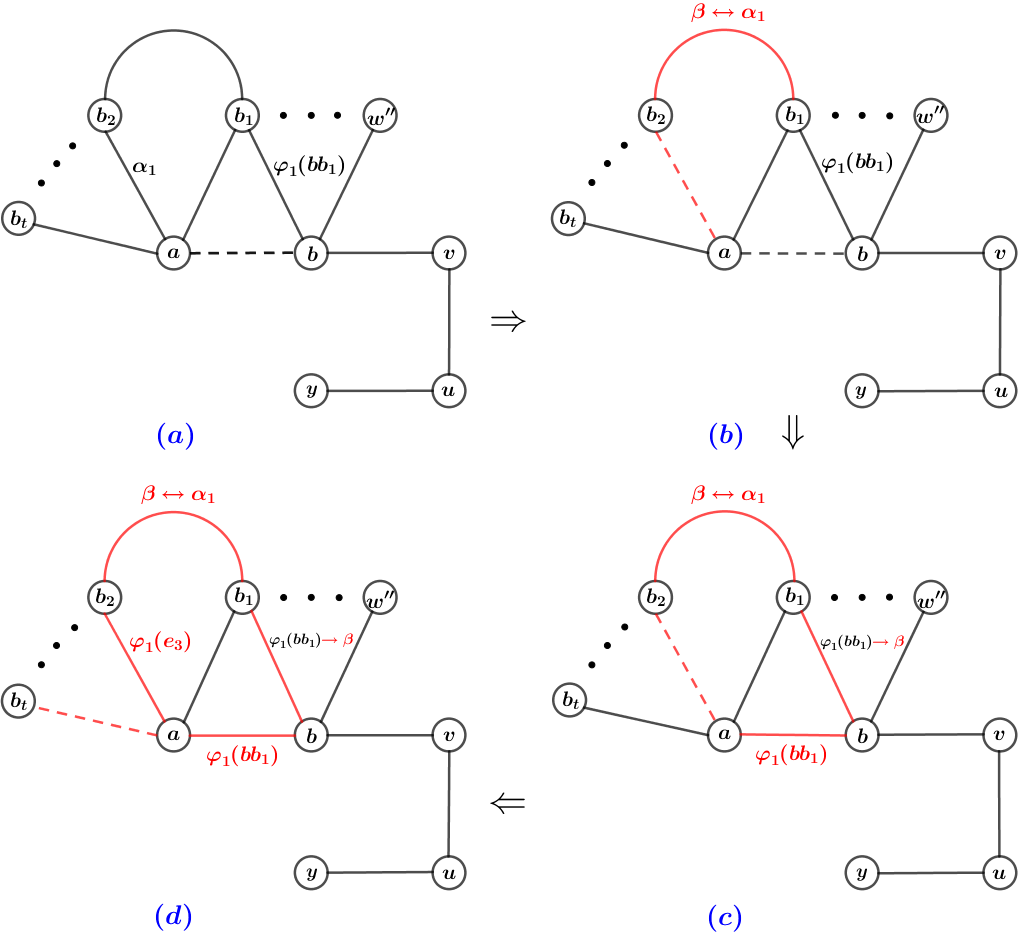}}
			\caption{One possible operation for $b_j=a^*\in(V(S')\backslash\{a\})\cap (V(S)\backslash\{a\})$ in Subcase 3.2, where $b_1=b_j=a^*=a'$. (The edges of the dashed line represent uncolored edges.)}
		\end{center}
	\end{figure}

Then we assume that there exists $b_j=a^*\in(V(S')\backslash\{a\})\cap (V(S)\backslash\{a\})$ for some $j\in [t-1]$ and $a^*\in V(S)$.
See Figure $3$ for a depiction when $b_1=b_j=a^*=a'$.
In this case we assume $a^*$ is the closest vertex to the vertex $a$ along $S$. Note that $b_j\neq b$ as $V(F'_a)\backslash\{a\}$ and $V(F_a)\backslash\{a\}$ are disjoint.  Let $\alpha_j=\varphi_1(e_{j+1})\in \overline{(\varphi_1)}_{H_1}(b_j)$. By Lemma \ref{MF}$(b)$, we have
$P_b(\beta,\alpha_j)=P_{b_j}(\beta,\alpha_j)$. If $e_{j+1}\notin P_b(\beta,\alpha_j)$, then we  apply a Kempe change on $P_{[b,b_{j}]}(\beta,\alpha_j)$, uncolor $e_{j+1}$ and color $e_{ab}$ with $\alpha_{j}$.
If  $e_{j+1}\in P_b(\beta,\alpha_{j})$ and $P_b(\beta,\alpha_{j})$ meets $b_{j+1}$ before $a$, then we apply a Kempe change on $P_{[b,b_{j+1}]}(\beta,\alpha_{j})$, uncolor $e_{j+1}$ and color $e_{ab}$ with $\alpha_{j}$. If  $e_{j+1}\in P_b(\beta,\alpha_{j})$ and $P_{b_{j}}(\beta,\alpha_{j})$ meets $b_{j+1}$ before $a$, then we uncolor $e_{j+1}$, apply a Kempe change on $P_{[b_i,b_{j+1}]}(\beta,\alpha_{j})$,  apply a shifting in $S$ from $a$ to $b_j$ (i.e., $a^*$), color  $e_{ab}$ with $\varphi_1(e_{ba'})$,  and recolor the edge $e_{bb_j}\in E_{H_1}(b,b_j)$ with $\beta$. (See Figure $3$(a)-(c).) In all three cases above, $e_{ab}$ is colored with a color in $[k]$ and $e_{j+1}$ is uncolored. Finally we apply a shifting  in $S'$  from $b_{j+1}$ to $b_t$, color $e_{j+1}$ with $\varphi_1(e_{j+2})$, and uncolor $e_t$. (See Figure $3$(d).) Notice that the above shifting in $S'$ does nothing if $b_{j+1}=b_t$. Denote $H_2:=H_1+e_{ab}-e_t$. Since $H_2$ is also $k$-dense and $\chi'(H_2)=k$, we can rename some color classes of $E(H_2)$ but keep the color $i$ unchanged to match all colors on boundary edges without producing any improper $i$-edge by Lemma \ref{lem-consisting}$(b)$.
Now we obtain a new matching $M^*_2:=(M^*_1\backslash \{e_{ab}\})\cup \{e_t\}$ of $G-V(M)$ and a new (proper) $k$-edge-coloring $\varphi_2$ of $G-(M\cup M^*_2)$ such that $f_{uv}$ is no longer T2-improper at $v$ but T1-improper  at $v$  with respect to the new prefeasible triple  $(M^*_2,\emptyset,\varphi_2)$. Furthermore, $|E_1(M_2^*,\varphi_2)|=|E_1(M_0^*,\varphi_0)|+2$ and $E_2(M_2^*,\varphi_2)=|E_2(M_0^*,\varphi_0)|-2$.	
Thus we can consider $(M^*_2,\emptyset,\varphi_2)$ instead.

	{\bf Subcase 3.3:}
$F_{b}$ does not contain a linear sequence at $b$ from $a$ to $y$ without $i$-edge, and $F_{b}$ does not contain a vertex $w''$ with $d_{H_1}(w'')=\Delta-1$ and $i\in\overline{(\varphi_1)}_{H_1}(w'')$.
	
We claim that $F_{b}$ contains a linear sequence $S^*$ at $b$ from $a$ to a $\Delta$-vertex $y^*$ such that $y^*\neq y$  and there is no $i$-edge in $S^*$. By Lemma \ref{MD}$(a)$, the multi-fan $F_b$ contains at least one $\Delta$-vertex in $H_1$.
Now if $F_{b}$ does not contain any linear sequence  without $i$-edges from $a$ to any $\Delta$-vertex in $H_1$, then by Lemma \ref{MD}$(c)$, the multi-fan $F_{b}$ contains a vertex $w''$ with $d_{H_1}(w'')=\Delta-1$ and $i\in\overline{(\varphi_1)}_{H_1}(w'')$, contradicting the condition of Subcase 3.3. So $F_{b}$ contains a linear sequence $S^*$ from $a$ to a vertex $y^*$ such that $d_{H_1}(y^*)=\Delta$ and there is no $i$-edge in $S^*$. Note that $y^*\neq y$, since otherwise we also have a contradiction to the condition of Subcase 3.3. Thus the claim is proved.
	
Assume that $S^*=(a,e_{ba'},a',\ldots, e_{by^*},y^*)$ at $b$ from $a$ to $y^*$ (where $a'=y^*$ is possible), and $S^*$ contains no $i$-edge. Let $\theta\in\overline{\varphi}_1(y^*)$.
	
	{\bf Subcase 3.3.1:} $\theta=i$.

Since $S^*$ contains no $i$-edge, we  apply a shifting  in $S^*$  from $a$ to $y^*$, color $e_{ab}$ with $\varphi_1(e_{ba'})$, uncolor $e_{by^*}$, and rename some color classes of $E(H_1+e_{ab}-e_{by^*})$ but keep the color $i$ unchanged to match all  colors on boundary edges without producing any improper $i$-edge by Lemma \ref{lem-consisting}$(b)$.  By coloring  $e_{by^*}$ with $i$ and recoloring  $e_{bv}$ from $i$ to $\Delta+\mu$,  we obtain a new matching $M^*_2:=M^*_1\backslash \{e_{ab}\}$ of $G-V(M)$  and a new (proper) $(k+1)$-edge-coloring $\varphi_2$ of $G-(M\cup M^*_2)$.  Then $f_{uv}$ is no longer  T2-improper  at $v$ or even T1-improper  at $v$ with respect to the new prefeasible triple $(M^*_2,E_{M_2^*}^{\varphi_2},\varphi_2)$ with  $E_{M_2^*}^{\varphi_2}= \{e_{bv}\}$  if $E_{M_1^*}^{\varphi_1}=\emptyset$, and $E_{M_2^*}^{\varphi_2}= \{e_{bv},h\}$  if $E_{M_1^*}^{\varphi_1}=\{h\}$ (when $y^*\in V(F_x)\cap V(F_b)$). Furthermore,
$E_{M_2^*}^{\varphi_2}\subseteq  (E_1(M_0^*,\varphi_0)\cup E_2(M_0^*,\varphi_0))$, $|E_1(M_2^*,\varphi_2)|\ge|E_1(M_0^*,\varphi_0)|$ and $|E_2(M_2^*,\varphi_2)|=|E_2(M_0^*,\varphi_0)|-2$.
Thus we can consider $(M^*_2,E_{M_2^*}^{\varphi_2},\varphi_2)$ instead.

	{\bf Subcase 3.3.2:} $\theta\neq i$.

Since $V(H_1)$ is $(\varphi_1)_{H_1}$-elementary, there exists a $\theta$-edge $e_1$ incident with the vertex $a$. Thus by the similar argument as in the proof of Subcase $3.2$, we  define a maximal multi-fan at $a$, denoted by  $F'_a$,  with respect to $e_1$ and $(\varphi_1)_{H_1}$ in $H_1+e_1$,
and we have $e_{F_a}(a,b')=e_{H_1+e_{ab}}(a,b')=\mu$ for any vertex $b'$ in $V(F_a)\backslash\{a\}$. Therefore, $V(F'_a)\backslash\{a\}$ and $V(F_a)\backslash\{a\}$ are disjoint, since otherwise we have $V(F'_a)\subseteq V(F_a)$ and $\varphi_1(e_1)=\theta\in \overline{(\varphi_1)}_{H_1}(b')$ for some $b'\in V(F_a)$ implying $y^*=b'\in V(F_a)$, which contradicts Assumption $(1)$. Note that $V(F'_a)\backslash \{a\}$ must contain a $\Delta$-vertex in $H_1$, since otherwise Lemma \ref{MD}$(d)$ and the fact $(\varphi_1)_{H_1}(e_1)=\theta\in\overline{\varphi}_1(y^*)$ imply that $y^*\in V(F'_a)$, which contradicts $d_{H_1}(y^*)=\Delta$.
If $F'_a$ contains a  vertex of $V(H_1)$ that is incident with an $i$-edge of $\partial_{G-(M\cup M_1^*)}(H_1)$ in $G-(M\cup M^*_1)$, then we denote the vertex by $w^*$ and the $i$-edge by $h^*$.
If $F'_a$ does not contain any linear sequence to a $\Delta$-vertex in $H_1$ without $i$-edge and boundary vertex $w^*$, then by Lemma \ref{MD}$(d)$, the multi-fan $F'_a$ contains a vertex $z^*$ with $i\in\overline{(\varphi_1)}_{H_1}(z^*)$ and $d_{H}(z^*)=\Delta-1$.  Since $H_1$ is $(\varphi_1)_{H_1}$-elementary, we have $z^*=w^*$ and $d_{H_1}(w^*)=\Delta-1$.
Thus $F'_a$ contains a linear sequence $S'=(b_1,e_2,b_2,\ldots, e_t,b_t)$  at $a$, where $b_1\in V(e_1)$, $b_t$ (with $t\ge 1$) is $w^*$ if there exists  $w^*$ with $d_{H_1}(w^*)=\Delta-1$ such that $h^*\in\partial_{G-(M\cup M_1^*)}(H_1)$ but $h^*\notin E_1(M_0^*,\varphi_0)$, and $b_t$ is a $\Delta$-vertex in $H_1$ otherwise. Notice that $b_t$ is not incident with any edge in $M\cup M^*_1$ by our choice of $b_t$.  Moreover, if $b_t=w^*$ as defined above, then $b_t=w^*$ is not a vertex in $V(F_b)$ by the condition of Subcase 3.3. And $b_t\neq y$ since $V(F'_a)\backslash\{a\}$ and $V(F_a)\backslash\{a\}$ are disjoint. Let $\beta$  ($\beta\neq i$) be a color in $\overline{\varphi}_1(b)$. By Lemma \ref{MF}$(b)$, we have $P_b(\beta,\theta)=P_{y^*}(\beta,\theta)$. We then consider the following two subcases according  the set $(V(S')\backslash\{a\})\cap (V(S^*)\backslash\{a\})$.

We first assume that $(V(S')\backslash\{a\})\cap (V(S^*)\backslash\{a\})\subseteq \{b_t\}$. If $e_1\notin P_b(\beta,\theta)$, then we  apply a Kempe change on $P_{[b,y^*]}(\beta,\theta)$, uncolor $e_1$ and color $e_{ab}$ with $\theta$.
If  $e_1\in P_b(\beta,\theta)$ and $P_b(\beta,\theta)$ meets $b_1$ before $a$, then we  apply a Kempe change on $P_{[b,b_1]}(\beta,\theta)$, uncolor $e_1$ and color $e_{ab}$ with $\theta$. If  $e_1\in P_b(\beta,\theta)$ and $P_{y^*}(\beta,\theta)$ meets $b_1$ before $a$, then we uncolor $e_1$, apply a Kempe change on $P_{[y^*,b_1]}(\beta,\theta)$, apply a shifting in $S^*$ from $a$ to $y^*$, color $e_{ab}$ with $\varphi_1(e_{ba'})$, and recolor $e_{by^*}$ with $\beta$. In all three cases above,  $e_{ab}$ is colored with a color in $[k]$ and $e_1$ is uncolored. Then we apply a shifting  in $S'$ from $b_1$ to $b_t$, color $e_1$ with $\varphi_1(e_2)$, and  uncolor $e_t$. Denote $H_2:=H_1+e_{ab}-e_t$. Since $H_2$ is also $k$-dense and $\chi'(H_2)=k$,  we can rename some color classes of $E(H_2)$ but keep the color $i$ unchanged to match colors on boundary edges except $i$-edges by Lemma \ref{lem-consisting}$(b)$. Finally recolor $h^*$ with the color $\Delta+\mu$ if $h^*\in\partial_{G-(M\cup M_0^*)}(H)\cap E_1(M_0^*,\varphi_0)$.
Now we obtain a new matching $M^*_2:=(M^*_1\backslash \{e_{ab}\})\cup \{e_t\}$ of $G-V(M)$  and
a new (proper) $(k+1)$-edge-coloring $\varphi_2$ of $G-(M\cup M^*_2)$ such that $f_{uv}$ is no longer T2-improper at $v$ but  T1-improper at $v$ with respect to the new prefeasible triple $(M^*_2,E_{M_2^*}^{\varphi_2},\varphi_2)$, where $\emptyset $ or $\{h\}$ or $\{h^*\}=E_{M_2^*}^{\varphi_2}\subseteq E_1(M_0^*,\varphi_0)$.
Furthermore,
$|E_1(M_2^*,\varphi_2)|\ge|E_1(M_0^*,\varphi_0)|$ and $|E_2(M_2^*,\varphi_2)|=|E_2(M_0^*,\varphi_0)|-2$. Thus we can consider $(M^*_2,E_{M_2^*}^{\varphi_2},\varphi_2)$ instead.
	
\begin{figure}[!ht]
		\begin{center}
			\centering
			\scalebox{0.34}{\includegraphics{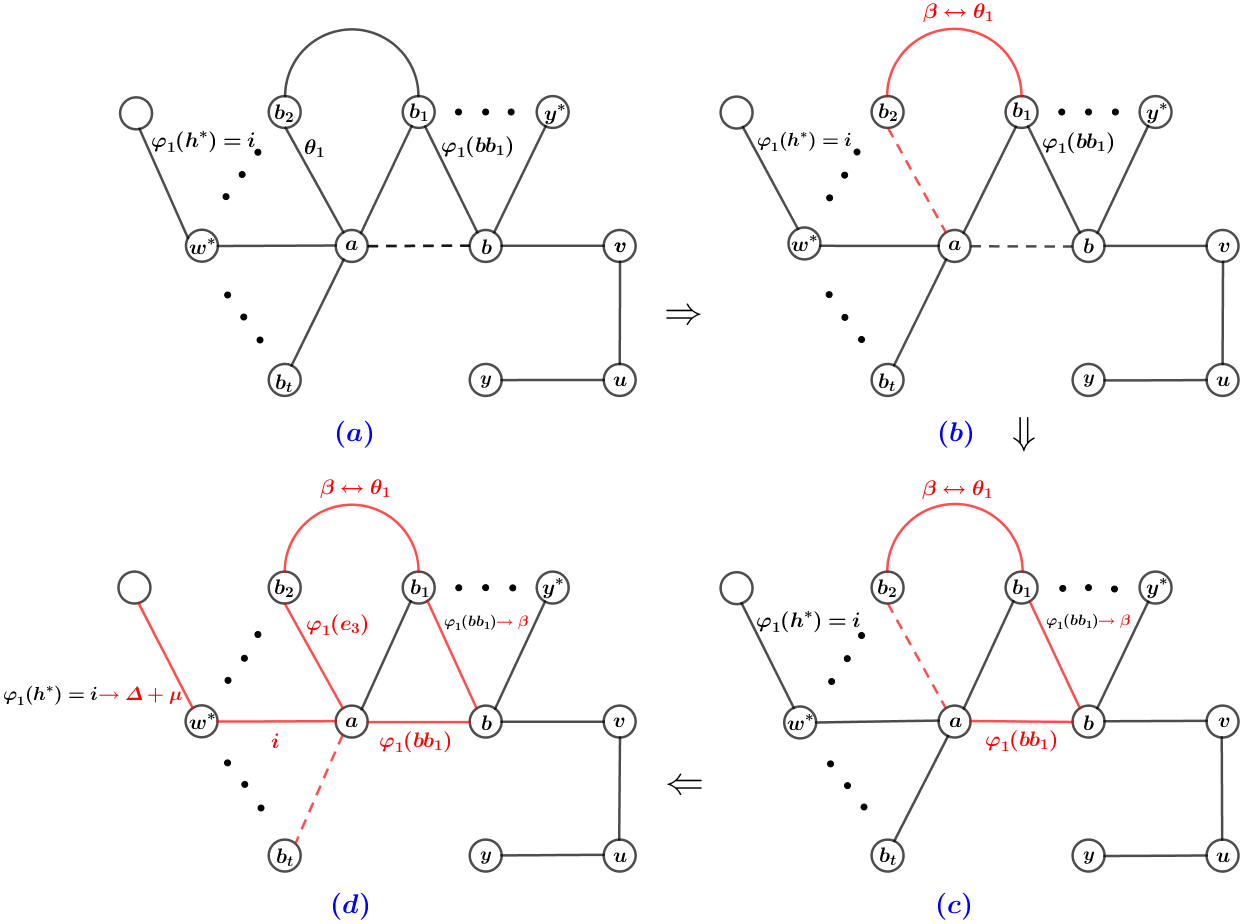}}
			\caption{One possible operation for $b_j=a^*\in(V(S')\backslash\{a\})\cap (V(S)\backslash\{a\})$ in Subcase 3.3, where $b_1=b_j=a^*=a'$. (The edges of the dashed line represent uncolored edges.)}
		\end{center}
	\end{figure}	
Then we assume that there exists $b_j=a^*\in(V(S')\backslash\{a\})\cap (V(S^*)\backslash\{a\})$ for some $j\in [t-1]$ and $a^*\in V(S^*)$.
See Figure $4$ for a depiction when $b_1=b_j=a^*=a'$.
In this case we assume $a^*$ is the closest vertex to $a$ along $S^*$. Note that $b_j\neq b$ as $V(F'_a)\backslash\{a\}$ and $V(F_a)\backslash\{a\}$ are disjoint.  Let $\theta_j=\varphi_1(e_{j+1})\in \overline{(\varphi_1)}_{H_1}(b_j)$. By Lemma \ref{MF}$(b)$, $P_b(\beta,\theta_j)=P_{b_j}(\beta,\theta_j)$. If $e_{j+1}\notin P_b(\beta,\theta_j)$, then we apply a Kempe change on $P_{[b,b_j]}(\beta,\theta_j)$, uncolor $e_{j+1}$ and color $e_{ab}$ with $\theta_{j}$.
	If  $e_{j+1}\in P_b(\beta,\theta_{j})$ and $P_b(\beta,\theta_{j})$ meets $b_{j+1}$ before $a$, then we apply a Kempe change on $P_{[b,b_{j+1}]}(\beta,\theta_{j})$, uncolor $e_{j+1}$ and color $e_{ab}$ with $\theta_{j}$. If  $e_{j+1}\in P_b(\beta,\theta_{j})$ and $P_{b_j}(\beta,\theta_{j})$ meets $b_{j+1}$ before $a$, then we uncolor $e_{j+1}$, apply a Kempe change on $P_{[b_j,b_{j+1}]}(\beta,\theta_{j})$,  apply a shifting in $S^*$ from $a$ to $b_j$ (i.e., $a^*$), color  $e_{ab}$ with $\varphi_1(e_{ba'})$,  and recolor the edge $e_{b{b_j}}\in E_{H_1}(b,b_j)$ with $\beta$. (See Figure $4$(a)-(c).) In all three cases above,  $e_{ab}$ is colored with a color in $[k]$ and $e_{j+1}$ is uncolored. Denote $H_2:=H_1+e_{ab}-e_t$. Then we apply a shifting in $S'$ from  $b_{j+1}$ to $b_t$, color $e_{j+1}$ with $\varphi_1(e_{j+2})$, and uncolor the edge $e_t$,  and rename some color classes of $E(H_2)$ but keep the color $i$ unchanged to match all colors on boundary edges  except $i$-edges by Lemma \ref{lem-consisting}$(b)$. Finally recolor $h^*$ with $\Delta+\mu$ if  $h^*\in\partial_{G-(M\cup M_0^*)}(H)\cap E_1(M_0^*,\varphi_0)$. (See Figure $4$(d).)
	Now we obtain a new matching $M^*_2=(M^*_1\backslash \{e_{ab}\})\cup \{e_t\}$ of $G-V(M)$ and a new (proper) $(k+1)$-edge-coloring $\varphi_2$ of $G-(M\cup M^*_2)$ such that $f_{uv}$ is no longer  T2-improper at $v$ but T1-improper at $v$  with respect to the new prefeasible triple $(M^*_2,E_{M_2^*}^{\varphi_2},\varphi_2)$, where $\emptyset $ or $\{h\}$ or $\{h^*\}=E_{M_2^*}^{\varphi_2}\subseteq E_1(M_0^*,\varphi_0)$. Furthermore,
$|E_1(M_2^*,\varphi_2)|\ge|E_1(M_0^*,\varphi_0)|$ and $|E_2(M_2^*,\varphi_2)|=|E_2(M_0^*,\varphi_0)|-2$.
Thus we can consider $(M^*_2,E_{M_2^*}^{\varphi_2},\varphi_2)$ instead. The proof is now finished.
\end{proof}

\end{document}